\documentclass[amsthm]{elsart}

\usepackage{yjsco}
\usepackage{natbib}

\usepackage{amsmath,amsthm,amssymb,graphicx}
\usepackage{algorithmic}
\usepackage{enumitem}
\usepackage{url}

\DeclareMathOperator{\lm}{lm}
\DeclareMathOperator{\lc}{lc}
\DeclareMathOperator{\NF}{NF}

\DeclareMathOperator{\spoly}{spoly}

\DeclareMathOperator{\Syz}{Syz}
\DeclareMathOperator{\LeftSyz}{LeftSyz}
\DeclareMathOperator{\Ann}{Ann}
\DeclareMathOperator{\ann}{Ann}

\DeclareMathOperator{\GKdim}{GK.dim}
\DeclareMathOperator{\gkdim}{GK.dim}
\DeclareMathOperator{\ini}{in}
\DeclareMathOperator{\codim}{codim}

\newcommand{\ra}{\rightarrow}

\newcommand{\N}{{\mathbb N}}
\newcommand{\R}{{\mathbb R}}
\newcommand{\K}{{\mathbb K}}
\newcommand{\Z}{{\mathbb Z}}

\newcommand{\C}{{\mathbb C}}

\renewcommand{\d}{\partial }
\newcommand{\BS}{Bernstein-Sato }
\newcommand{\BM}{Brian\c{c}on-Maisonobe }

\newcommand{\comment}[1]{}

\newcommand{\inw}{\ini_{(-w,w)}}
\newcommand{\inwh}{\ini_{(-w,w,0)}}
\renewcommand{\algorithmicrequire}{\textbf{Input:}} 
\renewcommand{\algorithmicensure}{\textbf{Output:}} 

\theoremstyle{definition}

\newtheorem{lemma}{Lemma}[section]
\newtheorem{proposition}[lemma]{Proposition}
\newtheorem{theorem}[lemma]{Theorem}
\newtheorem{corollary}[lemma]{Corollary}
\newtheorem{definition}[lemma]{Definition}
\newtheorem{example}[lemma]{Example}
\newtheorem{remark}[lemma]{Remark}

\newtheorem{algorithm}[lemma]{Algorithm}

\begin{document}
\setcounter{section}{0}

\begin{frontmatter}

\title{Effective Methods for the Computation of Bernstein-Sato polynomials for Hypersurfaces and Affine Varieties}

\author[Aachen]{Daniel Andres},
\ead{Daniel.Andres@math.rwth-aachen.de}
\author[Aachen]{Viktor Levandovskyy},
\ead{Viktor.Levandovskyy@math.rwth-aachen.de}
\author[Zaragoza]{Jorge Mart\'{i}n-Morales}
\ead{\\jorge@unizar.es}

\address[Aachen]{Lehrstuhl D f\"ur Mathematik, RWTH Aachen, Templergraben 64, 52062 Aachen, Germany}
\address[Zaragoza]{Department of Mathematics-I.U.M.A., University of Zaragoza, C/ Pedro Cerbuna, 12 - 50009, Zaragoza, Spain}

\begin{abstract}
This paper is the widely extended version of the publication, appeared in Proceedings
of ISSAC'2009 conference \citep*{ALM09}. We discuss more details on proofs, 
present new algorithms and examples.
We present a general algorithm for computing an intersection of a left ideal
of an associative algebra over a field with a subalgebra, generated by a single element. 
We show applications of this algorithm in different algebraic situations and describe our
implementation in \textsc{Singular}. Among other, we use
this algorithm in computational $D$-module theory for computing
e.~g. the Bernstein-Sato polynomial of a single polynomial with several approaches.
We also present a new method, having no analogues yet, for the
computation of the Bernstein-Sato polynomial of an affine variety.
Also, we provide a new proof of the algorithm by Brian\c{c}on-Maisonobe for the computation of the $s$-parametric annihilator of a polynomial. Moreover, we present new methods for
the latter computation as well as optimized algorithms for the computation of Bernstein-Sato polynomial in various settings.
\end{abstract}


\end{frontmatter}

\section{Introduction}

This paper extends \cite*{ALM09} by many details, proofs, algorithms and examples.
In this paper we continue reporting \citep{LVint, LM08, ALM09} on our advances in constructive $D$-module theory both in theoretical direction and also in the implementation, which we create in \textsc{Singular}. 

Our work on the implementation of procedures for $D$-modules started in 2003, motivated among other factors by challenging elimination problems in non-commutative algebras,
which appear e.~g. in algorithms for computation of Bernstein-Sato polynomials. We reported on solving several challenges in \cite{LM08}.
A non-commutative subsystem \textsc{Singular:Plural} \citep{Plural} of the
computer algebra system \textsc{Singular} provides a user with possibilities
to compute numerous Gr\"obner bases-based procedures in a wide class 
of non-commutative $G$-algebras \citep{LS03}. It is natural to use
this functionality in the context of computational $D$-module theory.

As of today, the $D$-module suite in \textsc{Singular} consists of three
libraries: \texttt{dmod.lib}, \texttt{dmodapp.lib} and \texttt{bfun.lib}.
Moreover, \texttt{gmssing.lib} \citep{Gmssinglib} contains some sophisticated (and hence 
fast) and useful procedures, e.~g. \texttt{bernstein} for the computation of
the local \BS polynomial of an isolated singularity at the origin. 
There are many useful and flexible procedures for various aspects of $D$-module theory. These libraries are freely distributed together with \textsc{Singular} \citep{Singular} since the version 3-1-0, which was released in April 2009.
More libraries are currently under development, among them procedures for computing
the restriction, integration and localization of $D$-modules.

There are several implementations of algorithms for $D$-modules, namely the
experimental program \textsc{kan/sm1} by N.~Ta\-ka\-ya\-ma \citep{KAN}, 
the \texttt{bfct} package in \textsc{Risa/Asir} \citep{Asir} by M.~Noro \citep{Noro02}
and the package \texttt{Dmodules.m2} in \textsc{Macaulay2} by A.~Leykin and H.~Tsai \citep{dmodMac2}. To the best of our knowledge, there is ongoing work by the \cite{cocoa} to develop some $D$-module functionality as well.
We aim at creating a $D$-module suite, which will combine flexibility and rich functionality with high performance, being able to treat more complicated examples. 

We continue comparing our implementation (cf. Section \ref{compare}) with the ones in the systems \textsc{Asir} and \textsc{Macaulay2}, see \cite{LM08} for earlier results.

In this paper, we address the following computational problems:


\begin{itemize}
\item $s$-parametric annihilator of $f\in\K[x_1,\ldots,x_n]$
\item Bernstein-Sato ideals for $f = f_1 \cdot \ldots \cdot f_m$
\item $b$-function with respect to weights for an ideal in $D$
\item Global Bernstein-Sato polynomial of $f$ 
\item Bernstein-Sato polynomial for a variety
\end{itemize}


In Section \ref{newBM}, we give a new proof for the algorithm by \BM  for computing $\ann_{D[s]} f^s$, announced in \cite{LM08}. 
Moreover, using the same technique we design a new algorithm for the computation of the \BS polynomial of an affine variety, following the paper \cite{BMS06}, and prove its correctness.

We develop the method of principal intersection \ref{PrincipalIntersect} in the
general context of $\K$-algebras and discuss its improvements. This algorithm
is especially useful for problems of $D$-module theory, since it allows to replace
a generally hard elimination with Gr\"obner bases by the search for a $\K$-linear dependence
of a sequence of normal forms. The algorithm 
is applied
in Section \ref{globalBS} to two main methods for computing \BS polynomials as well as
to solving 0-dimensional systems in commutative rings and to the computation
of central characters in Section \ref{PIapps}. 
Moreover, we describe a folklore method for computing \BS polynomial via annihilator (using, however, principal intersection instead of Gr\"obner-based elimination) and prove in Lemma \ref{compareReducedB}, that is is more efficient than the usual one.

The generalization of principal intersection approach
to the case of more general subalgebras we discuss in Section \ref{Imulti}.


\subsection{Notations}
Throughout the article $\K$ stands for a field of characteristic zero. By 
$R$ we denote the polynomial ring $\K[x_1,\ldots,x_n]$ and by $f\in R$ a
non-constant polynomial.

We consider the $n$-th Weyl algebra as the algebra of linear partial differential
operators with polynomials coefficients. That is $D_n = D_n(R) = \K\langle x_1,\dots,x_n,\partial_1,\dots,\partial_n \mid \{ \partial_i x_i = x_i \partial_i +1, \partial_i x_j = x_j \partial_i, i \not=j \}  \rangle$. 
We denote by $D_n[s] = D_n(R) \otimes_{\K} \K[s_1,\ldots,s_n]$ and drop the index $n$ depending on the context. 

The ring $R$ is a natural $D_n(R)$-module with the action
\[
x_i \bullet f(x) = x_i \cdot f(x), \ 
\d_i \bullet f(x) = \frac{\d f(x)}{\d x_i}. 
\]

Working with monomial orderings in elimination, we use the notation $x \gg y$ for ``$x$ is greater than any power of $y$''.

Given an associative $\K$-algebra $A$ and some monomial well-ordering on $A$,
we denote by $\lm(f)$ (resp. $\lc(f)$) the leading monomial (resp. the leading coefficient)
of $f\in A$. Given a left Gr\"obner basis $G \subset A$ and $f\in A$, 
we denote by $\NF(f,G)$ the normal form of $f$ with respect to the left ideal ${}_{A}\langle G \rangle$. We also use the shorthand notation $h \ra_{H} f$ (and $h \ra f$, if $H$ is clear
from the context) for the reduction of $h\in A$ to $f\in A$ with respect to the set $H$. 
If not specified, under \textit{ideal} we mean \textit{left ideal} in a $\K$-algebra,
and by \textit{Gr\"obner basis} a \textit{left Gr\"obner basis}.
For $a,b$ in some $\K$-algebra $A$, we use the Lie bracket notation $[a,b]:=ab - ba$
as well as skew Lie bracket notation $[a,b]_{k}:=ab - k\cdot ba$ for $k\in\K^{*}$.

We say, that a proper subalgebra $S$ of an associative $\K$-algebra $A$ is a 
\textit{principal subalgebra}, if there exists $g\in A\setminus \K$, such that $S = \K[g]$.

Let $M$ be an $A$-module, then we denote the \textit{Gel'fand-Kirillov dimension}
of $M$ (see \cite{MR} for the details and \cite{BGV} for algorithms)  by $\GKdim (M)$.
Recall, that a module $M$ is called \textit{holonomic}, if
$\GKdim (A/\ann M) = 2 \cdot \GKdim(M)$. We prefer this general definition, since it
concides with the classical way of defining holonomy in Weyl algebras and it
is incomparably more general to the latter. In particular, armed with this general
definition we can speak on holonomic modules over any $G$-algebra.

\section{Preliminaries}

It is convenient to treat the algebras we deal with in the bigger framework of $G$-algebras
of Lie type.

\begin{definition}
\label{GalgLie}
Let 
 $A$ be the quotient of the free associative algebra 
$\K\langle x_1,\ldots,x_n\rangle$ by the two-sided ideal $I$, generated by
the finite set $\{x_jx_i - x_ix_j-d_{ij} \} \; \forall 1\leq i<j \leq n$, where
$d_{ij} \in \K[x_1,\ldots,x_n]$. The algebra $A$ is called a {\em $G$--algebra of Lie type} \citep{LS03}, if
\begin{itemize}
\item $\forall \; 1\leq i < j < k \leq n$ the expression
$d_{ij}x_k - x_k d_{ij} + x_j d_{ik} - d_{ik}x_j + d_{jk}x_i - x_i d_{jk}$
reduces to zero modulo $I$ and
\item there exists a monomial ordering $\prec$ on $\K[ x_1,\ldots,x_n]$,
such that $\lm(d_{ij}) \prec x_i x_j$ for each $i < j$.
\end{itemize}	
\end{definition}

These algebras were also studied in \cite{KW90} and \cite{BGLC98} by the names PBW algebras and algebras of solvable type.

Recall the algorithm for computing the preimage of a left ideal under a homomorphism 
of $G$-algebras from \cite{LVint}. 

\begin{theorem}[Preimage of a Left Ideal, \cite{LVint}] \label{ncPreimage} \hfill\\ 
Let $A,B$ be $G$-algebras of Lie type, generated by $\{x_i \mid 1\leq i \leq n\}$ and 
$\{y_j \mid 1\leq j \leq m\}$ respectively, subject to finite sets of relations $R_A, R_B$ as in Def. \ref{GalgLie}. Let $\phi: A \to B$ be a homomorphism of $\K$-algebras.
 Define $I_\phi$ to be the $(A,A)$-bimodule 
 ${}_{A}\langle \{ x_i - \phi(x_i) \mid 1\leq i \leq n\} \rangle_{A} \subset A \otimes_{\K} B$.
Suppose, that there exists an elimination ordering for $B$ on $A \otimes_{\K} B$, satisfying the following conditions 
\[
 1\leq i \leq n, 1\leq j \leq m, \qquad  \lm(y_j \phi(x_i) - \phi(x_i) y_j)  \prec x_i y_j.
\] 

Then there are the following statements.\\
\textbf{1)} Define $A \otimes^{\phi}_{\K} B$ to be the $\K$-algebra generated by $\{ x_1,\ldots, x_n, y_1, \ldots, y_m \}$ subject to the finite set of relations composed of $R_A, R_B$ and 
$\{y_j x_i - x_i y_j -  y_j \phi(x_i) + \phi(x_i) y_j \}$. Then $A \otimes^{\phi}_{\K} B$ is a $G$-algebra of Lie type.\\
\textbf{2)} Let $J \subset B$ be a left ideal, then
\[
\phi^{-1}(J)=(I_\phi+J) \cap A \ \subset \ A \otimes^{\phi}_{\K} B \cap A.
\]
Moreover, this computation can be done by means of elimination.
\end{theorem}

The following proposition is a reformulation
of Theorem \ref{ncPreimage}, adopted to the situation, which is often encountered in context
of $D$-modules.

\begin{proposition}
\label{preimageCor}
Let $A_1, B_1, C$ be $G$-algebras of Lie type 
and $\varphi: A_1 \ra B_1$ be a homomorphism of $\K$-algebras. 
Consider the following data:
$$
\begin{array}{c}
A = C\otimes_\K A_1,\quad B = C\otimes_\K B_1,\quad \phi = 1_{C} \otimes\varphi: A \ra B,\\
E = A\otimes_\K^{\phi} B,\qquad  E' = C\otimes_\K (A_1\otimes_\K^{\varphi} B_1).\\
\end{array}
$$
Then $A \subset E' \subset E$ and for a left ideal $J\subset B$ we have:
\begin{enumerate}
\item $(E I_\varphi + E J) \cap E' = E'I_{\varphi} + E'J$.
\item $\phi^{-1}(J) = (E'I_{\varphi} + E'J)\cap A$.
\end{enumerate}
Moreover, the second intersection can be computed using Gr\"obner bases, provided
there exists an elimination ordering for $B_1$ on $E'$ compatible with
the $G$-algebra structure of $E'$.
\end{proposition}

In the proofs we quite often use the following.

\begin{lemma}[Generalized Product Criterion, \cite{LS03}]
\label{prodCrit}
Let $A$ be a $G$-algebra of Lie type and $f,g\in A$. Suppose that
$\lm(f)$ and $\lm(g)$ have no common factors, then $\spoly(f,g)$ reduces to $[f,g]$ with respect to the set $\{ f,g \}$.
\end{lemma}

\section{$s$-parametric Annihilator of $f$} \label{sAnn}

Recall Malgrange's construction for $f = f_1 \cdot \ldots \cdot f_p\in R=\K[x_1,\ldots,x_n]$.
Let $D_n = D_n(R)$, $T=\K[t_1,\ldots,t_p]$ and $D'_p:= D_p(T)$.
Consider the $(p+n)$-th Weyl algebra $D_{p+n} = D_n \otimes_{\K} D'_p$.
Moreover, consider
the following left ideal in $D_{p+n}$, called the {\em Malgrange ideal}
\[
I_f  := \langle \; \{ \; t_j - f_j, \partial_i + \sum^p_{j=1} \frac{\partial f_j}{\partial x_i} {\partial t}_j, \; 1\leq j\leq p, 1\leq i\leq n \;\}\; \rangle.
\]
Then for $s=(s_1,\ldots,s_p)$ we denote $f^s := f_1^{s_1} \dots f_p^{s_p}$. Let us compute
\[
 I_f \cap \K[\{t_j \d t_j\}]\langle x_i, \d x_i \mid [\d x_i, x_i]=1 \rangle \subset D_n[\{t_j \d t_j \}] \subset D_{p+n}
\]
and furthermore, replace $t_j \d t_j$ with $-s_j-1$. The result is known (e.~g. \cite{SST00}) to be exactly $\Ann_{D[s]} f^s \subset D[s]$.

There exist several methods for the computation of the $s$-parametric annihilator of~$f^s$.

\subsection{Oaku and Takayama}

The algorithm of Oaku and Takayama \citep{Oaku97a,Oaku97b,Oaku97,SST00} was developed in a wider context
and uses homogenization. With notations as above, let
$H := D_n \otimes _\K D'_p \otimes _\K \K[u_1, \ldots, u_p, v_1, \ldots, v_p].$
Moreover, let $I$ below be the $(u,v)$-homogenized Malgrange ideal, that is the left ideal in $H$
\[
I = \left\langle  \{t_j - u_j f_j, \partial_i + \sum^p_{k=1} \frac{\d f_k}{\d x_i} u_k {\d t}_j, u_j v_j -1 \} \right\rangle.
\]
Oaku and Takayama proved, that $\ann_{D_n[s]}(f^s)$ can be obtained in two steps.
At first $\{u_j, v_j\}$ are eliminated from $I$ with the help of Gr\"obner bases, thus
yielding $I' = I \cap (D_n \otimes _\K D'_p)$. Then, one calculates
$I' \cap (D_n \otimes _\K \K[\{-t_j \d t_j -1 \}])$ and
substitutes every appearance of $t_j \d t_j$ by $-s_j-1$ in the latter. 

\subsection{Brian\c{c}on and Maisonobe}
\label{BM}

Consider $S_p = \K\langle \{ \d t_j, s_j \} \mid \d t_j s_k = s_k \d t_j - \delta_{jk} \d t_j  \rangle$ (the $p$-th shift algebra) and $B = D_n \otimes_{\K} S_p$. Moreover, consider the following left ideal in $B$:
\[
I = \left\langle \{s_j +  f_j \d t_j, \partial_i + \sum^p_{k=1} \frac{\partial f_k}{\partial x_i} {\partial t}_k \} \right\rangle.
\]
Brian\c{c}on and Maisonobe proved in \cite{BM02} that $\ann_{D_n[s]}(f^s) = I\cap D_n[s]$ and 
hence the latter can be computed via the left Gr\"obner basis
with respect to an elimination ordering for $\{\d t_j\}$.


\subsection{A new proof for Brian\c{c}on-Maisonobe}
\label{newBM}

By using the Preimage Theorem \ref{ncPreimage} we give a new, completely computer-algebraic 
proof for the method of \BM \ref{BM}.

Let $A := D_n[s] = \K\langle \{ s_j, x_i, \d_i \} \mid \d_i x_i=x_i \d_i+1 \rangle$, and
\[
B := \K\langle \{t_j, \d t_j, x_i, \d_i \}\mid \{\d_i x_i=x_i \d_i+1,  \d t_j t_j=t_j \d t_j+1\} \rangle.
\]

Thus in the notations of Proposition \ref{preimageCor}, 
$C = D_n, A_1 = \K[s], B_1 = \K\langle \{t_j, \d t_j \mid \d t_j t_j=t_j \d t_j+1\} \rangle$
and $A= C \otimes A_1, B= C \otimes B_1$. Consider the algebraic Mellin transform (cf. \cite{SST00}) $\varphi: A_1 \to B_1$, $\ s_j \mapsto -t_j \d t_j-1$.


Hence $I_\varphi = \langle \{t_j \d t_j + s_j + 1\} \rangle \subset A_1\otimes^{\varphi}_{\K} B_1 =: E'$. 
Since $[t_k, s_j] = \delta_{jk} t_j$ and $[\d t_k, s_j] = -\delta_{jk} \d t_j$, the ordering conditions of Theorem \ref{ncPreimage} take the form $t_j \prec s_j t_j, \d t_j \prec s_j \d t_j$, which are satisfied if and only if $1 \leq t_j, \d t_j, s_j$. 

By Proposition \ref{preimageCor}, 
for any $L \subset B$, $\phi^{-1}(L)=(I_\phi+L) \cap A$. 
Hence, 
\begin{align*}
I_\phi+L 
&= \langle \{t_j - f_j, \partial_i + \sum^p_{j=1} \frac{\partial f_j}{\partial x_i} {\partial t}_j,  t_j \partial t_j + s_j + 1 \} \rangle\\
&= \langle \{t_j - f_j, \partial_i + \sum^p_{j=1} \frac{\partial f_j}{\partial x_i} {\partial t}_j,  f_j \partial t_j + s_j \} \rangle
\end{align*}
because $t_j \d t_j + s_j + 1$ reduces to
\[
 t_j \d t_j + s_j + 1 - \d t_j \cdot (t_j - f_j) = f_j \d t_j + s_j 
 \in I_{\phi} + L.
\]


\begin{lemma}
\label{BMgb}
Consider an ordering $\prec_T$, which satisfies the property 
$\{t_j \}  \gg  \{ x_i \} $, $\{\d_i, s_j \}  \gg  \{ x_i, \d t_j \}$. 
Moreover, set
\[
g_i :=\partial_i + \sum^p_{k=1} \frac{\partial f_k}{\partial x_i} {\partial t}_k, \
S_1 := \{ t_j - f_j, g_i \}, \
S_2 := S_1 \cup \{s_j +  f_j \d t_j \} \subset E'.
\]
Then $S_1$ and $S_2$ are left Gr\"obner bases with respect to $\prec_T$.
\end{lemma}

\begin{proof}
We run Buchberger's algorithm by hand. There are only three kinds of critical pairs we have to consider. Due to the ordering property, for each pair the generalized Product Criterion is applicable. Hence, we need to compute just the Lie brackets of members of pairs.
\begin{enumerate}[label=\arabic{*}.]
\item $[t_i - f_i,t_k - f_k] = 0$.
\item 
For pairs $(g_i,g_k)$ computing $[g_i,g_k]$ yields
\[
\sum_j  {\d t}_j [\d_i, \frac{\d f_j}{\d x_k} ] +  \sum_j {\d t}_j  [\frac{\d f_j}{\d x_i}, \d_k]
= \sum_j {\d t}_j ([\d_i, \frac{\d f_j}{\d x_k} ] - [\d_k, \frac{\d f_j}{\d x_i} ]). 
\]
Since $[\d_i, \frac{\d f_j}{\d x_k}] = \frac{\d^2 f_j}{\d x_i x_k} = [\d_k, \frac{\d f_j}{\d x_i} ]$, $\spoly(g_i,g_k)$ reduces to zero.
\item For mixed pairs $(t_k - f_k, g_i)$ we have
\[
[t_k - f_k,g_i] = \sum_j \frac{\d f_j}{\d x_i} [t_k,\d t_j] - [f_k, \d_i]=0.
\]
 Hence, $S_1$ is a left Gr\"obner basis. Now, 
in $S_2$ there are three new kinds of critical pairs to consider and 
for all of them we can apply the generalized Product Criterion.
\item $[t_k - f_k, s_j +  f_j \d t_j]
  = [t_k,s_j] + f_j [t_k, \d t_j] - [f_k,s_j]- [f_k, f_j \d t_j] 
  = \delta_{jk} (t_k -f_j) \to 0.$
\item $[s_i +  f_i \d t_i , s_j +  f_j \d t_j] = f_j [s_i, \d t_j] - f_i [s_j, \d t_i] = 0$.
\item Finally,
\begin{align*}
&\quad[s_j +  f_j \d t_j,  \d_i + \sum\limits^p_{k=1} \frac{\d f_k}{\d x_i} {\d t}_k]\\
&= [s_j, \d_i]  + \d t_j [f_j, \d_i]
 + \sum\limits^p_{k=1} \frac{\d f_k}{\d x_i} [s_j, {\d t}_k]
 + [f_j \d t_j, \sum\limits^p_{k=1} \frac{\d f_k}{\d x_i} {\d t}_k]\\
 &= \frac{\d f_j}{\d x_i} {\d t}_j - [\d_i, f_j] \d t_j  = 0.
\end{align*}
\end{enumerate}
So, $S_2$ is a left Gr\"obner basis.
\end{proof}
We want to eliminate both $\{t_j\}$ and $\{\d t_j\}$ from $I_{\phi} + L$. As we see above, by using an
elimination ordering for $\{t_j\}$ we proved above that $S_2$ is a Gr\"obner basis.
So, the elimination ideal is generated by $S_3 := S_2 \setminus \{t_i - f_i\}$. Hence we can proceed with eliminating $\{\d t_j\}$ from $S_3$, which is exactly the statement of the \BM algorithm in Section \ref{BM}.

\subsection{Bernstein-Sato ideals for $f = f_1 \cdot \ldots \cdot f_m$}
\label{BSideal}

Comparing the effectiveness of the algorithms, \cite{GHU05} concluded that the 
method of \BM is the best for the computation of $s$-parametric annihilators.
In \cite{LM08} we gave experimental results for the case $f = f_1$ and showed,
that the algorithm of \BM is faster than the LOT method, which in turn is
faster than the algorithm of Oaku and Takayama.

Because of the structure of annihilators in the situation $f = f_1 \cdot \ldots \cdot f_p$, $p>1$,
basically the same principles stand behind the corresponding algorithms. 

Let $s = (s_1,\ldots,s_p)$, then a \textit{\BS ideal} in $\K[s]$, which is defined
as 
\[
\mathcal{B}(f) = 
(\Ann_{D[s_1,\ldots,s_p]} f_1^{s_1} \cdot \ldots \cdot f_p^{s_p} + \langle f_1 \cdot \ldots \cdot f_p \rangle ) \cap \K[s_1,\ldots,s_p],
\]
can be computed with the help of $\Ann_{D[s]} f^s \subset D[s]$.
See \cite{Bahloul01} for algorithms.
In contrary to the case $f=f_1$, the ideal $\mathcal{B}(f)$ need not be principal in general. However, it is an open question to give a criterion for the principality of $\mathcal{B}(f)$. 
Armed with such a criterion, one can apply the method of Principal Intersection \ref{PrincipalIntersect} and thus replace expensive elimination above by the computation of a minimal polynomial.
As in the case $f=f_1$ it is an open question, which strategy and which orderings should one
use in the computation of the annihilator in order to achieve better performance.

\subsection{Implementation}

Due to the comparison above, we decided to implement only \BM \ method for the $(s_1,\ldots,s_p)$-para\-met\-ric annihilator $\Ann_{D[s]} f^s \subset D[s]$ in the case of $p>1$.

The corresponding procedure of \texttt{dmod.lib} is called \texttt{annfsBMI}.
It computes both annihilator and the \BS ideal.

We reported in \cite{LM08} on several computational challenges, which have been solved with
the help of our implementation.


We use the following acronyms in adressing functions in the implementation:
\textit{OT} for Oaku and Takayama, \textit{LOT} for Levandovskyy's
modification of Oaku and Takayama \citep{LM08} and \textit{BM} for Brian\c{c}on-Maisonobe. Moreover, it is possible to specify the desired Gr\"obner basis engine (\texttt{std} or \texttt{slimgb}) via an optional argument. 

For the classical situation where $f = f_1$, there are \texttt{SannfsOT, SannfsLOT, SannfsBM} procedures implemented, each along the lines of the corresponding algorithm. Moreover, there is
a procedure \texttt{Sannfs}$(f)$, computing $\Ann_{D[s]} f^s \subset D[s]$ using a ``minimal user knowledge'' principle. 

\begin{example}
\label{exSannfs}
We demonstrate, how to compute the $s$-para\-met\-ric annihilator with \texttt{Sannfs}. This procedure takes a polynomial in a commutative ring as its argument and returns back a Weyl algebra of the type \texttt{ring} together with an object of the type \texttt{ideal} called \texttt{LD}. 

\begin{verbatim}
LIB "dmod.lib";
ring r = 0,(x,y),dp;        // set up commutative ring
poly f = x^3 + y^2 + x*y^2; // define polynomial 
def D = Sannfs(f);          // call Sannfs
setring D; LD;              // activate ring D, print Ann(f^s)
==> LD[1]=2*x*y*Dx-3*x^2*Dy-y^2*Dy+2*y*Dx
==> LD[2]=2*x^2*Dx+2*x*y*Dy+2*x*Dx+3*y*Dy-6*x*s-6*s
==> LD[3]=x^2*y*Dy+y^3*Dy-2*x^2*Dx-3*x*y*Dy-2*y^2*s+6*x*s
\end{verbatim}
Note, that \texttt{LD} is not a Gr\"obner basis but a set of generators. Computing 
a Gr\"obner basis is done by
\verb?groebner(LD);?. In this case \texttt{groebner} returns the generators above and 2 new ones.
\end{example}


\section{$b$-functions with respect to weights for an ideal} \label{bfctIdeal}

Let $0 \neq w \in \R^n_{\geq 0}$ and consider the $V$-filtration
$V = \left\{ V_m \mid m \in \Z \right\}$
on $D$ with respect to $w$,  where 
$V_m$ is spanned by $\left\{ x^{\alpha} \d^{\beta} \mid -w \alpha + w \beta \leq m \right\}$ over $\K$.
That is, $x_i$ and $\d_i$ get weights $-w_i$ and $w_i$ respectively. Note, that with respect to such weights the relation $\d_i x_i = x_i \d_i + 1$ is homogeneous of degree $0$.
It is known that the associated graded ring $\bigoplus_{m \in \Z} V_m / V_{m-1}$ is isomorphic to $D$, which allows us to identify them.

From now on we assume, that $I$ is an ideal such that $D/I$ is a holonomic module. Since holonomic $D$-modules are cyclic (e.~g. \cite{Cou95}), for each holonomic $D$-module $M$ there exists an ideal $I_M$ such that $M \cong D/I_M$ as $D$-modules.

\begin{definition}
Let $0 \neq w \in \R^n_{\geq 0}$. For non-zero
\[
	p = \sum_{\alpha, \beta \in \N_0^n} c_{\alpha \beta} x^{\alpha} \d^{\beta} \in D
\]
we put $m = \max_{\alpha, \beta} \{ -w \alpha + w \beta \mid c_{\alpha \beta} \neq 0\} \in \R$
and define the \emph{initial form of $p$ with respect to the weight $w$} as follows:
\[
	\inw(p) := \sum_{\alpha,\beta \in \N_0^n:~ -w \alpha + w \beta = m} c_{\alpha \beta} x^{\alpha} \d^{\beta}.
\]
For the zero polynomial, we set $\inw(0) := 0$.
Additionally, we call the graded ideal
$\inw(I) := \K \cdot \{ \inw(p) \mid p \in I \}$
the \emph{initial ideal of $I$ with respect to the weight $w$}.
\end{definition}

\begin{definition}
Let $0 \neq w \in \R^n_{\geq 0}$ and $s := \sum_{i=1}^n w_i x_i \d_i$. 
Then $\inw(I) \cap \K[s]$ is a principal ideal in $\K[s]$.
Its monic generator $b(s)$ is called the \emph{global $b$-function of $I$ with respect to the weight $w$}.
\end{definition}

\begin{theorem} \label{b(s) is not zero}
The global $b$-function of $I$ is nonzero.
\end{theorem}
We will give yet another proof of this theorem \citep{SST00} in Section \ref{intersection}.

Note, that by setting the weight vector in an appropriate way, one can compute
$b$-functions of holonomic $D$-modules $D/I$, which are usually referred as
$b$-function for restriction, integration, localization etc. These special
$b$-functions play an important role in the computation of the corresponding
restriction, integration, localization modules, see \cite{Oaku97, SST00}.

Following its definition, the computation of the global $b$-function of $I$ with respect to $w$ can be done in two steps:\\
1. Compute the initial ideal $I'$ of $I$ with respect to $w$.\\
2. Compute the intersection of $I'$ with the subalgebra $\K[s]$.

We will discuss both steps separately, starting with the initial ideal.

\subsection{Computing the initial ideal} \label{initial ideal}

In order to compute the initial ideal, the method of weight\-ed homogenization has been proposed in \cite{Noro02}. A more general approach on homogenization of differential operators can be found in \cite{CN97}.

Let $u,v \in \R^n_{>0}$. The associative $\K$-algebra
$D_{(u,v)}^{(h)} 
$ is a $G$-algebra in the variables $x_1,\ldots,x_n,\d_1,\ldots,\d_n,h$ which commute pairwise except for
$\d_j x_i = x_i \d_j + \delta_{ij} h^{u_i + v_j}$.  
$D_{(u,v)}^{(h)}$ is called the \emph{$n$-th weighted homogenized Weyl algebra} with weights $u,v$,
i.e. $x_i$ and $\d_i$ get weights $u_i$ and $v_i$ respectively.

For $p = \sum_{\alpha, \beta} c_{\alpha \beta} x^{\alpha} \partial^{\beta} \in D$
one defines the \emph{weighted homogenization} of $p$ as follows:
\[
	H_{(u,v)}(p) = \sum_{\alpha, \beta} c_{\alpha \beta} h^{\deg_{(u,v)}(p) - (u \alpha + v \beta)} x^{\alpha} \d^{\beta}.
\]
This definition naturally extends to a set of polynomials.
Here, $\deg_{(u,v)}(p)$ denotes the weighted total degree of $p$ with respect to weights $u,v$ for $x,\d$ and weight $1$ for $h$.

For a monomial ordering $\prec$ in $D$, which is not necessarily a well-ordering, we define an associated homogenized global ordering $\prec^{(h)}$ in $D_{(u,v)}^{(h)}$ by setting $h \prec^{(h)} x_i, h \prec^{(h)} \d_i$ for all $i$ and,
\begin{alignat*}{2}
	p \prec^{(h)} q
		&\quad\text{if}\quad &\deg_{(u,v)}(p) &< \deg_{(u,v)}(q)\\
		&\quad\text{or}\quad &\deg_{(u,v)}(p) &= \deg_{(u,v)}(q) 
			 \quad\text{and}\quad
			p_{\mid_{h=1}} \prec q_{\mid_{h=1}}.
\end{alignat*}

Note that for $u = v = (1,\ldots,1)$ this is exactly the standard homogenization as in \cite{SST00}.
Analogue statements of the following two theorems can be found in \cite{SST00} and \cite{Noro02} respectively.
Due to our different conception of Gr\"obner bases (we require well-orderings), we give new proofs for them.

\begin{theorem} \label{Theorem: dehomogenization}
Let $F$ be a finite subset of $D$ and $\prec$ a global ordering.
If $G^{(h)}$ is a Gr\"obner basis of $\langle H_{(u,v)}(F) \rangle$ with respect to $\prec^{(h)}$,
then ${G^{(h)}}_{\mid_{h=1}}$ is a Gr\"obner basis of $\langle F \rangle$ with respect to $\prec$.
\end{theorem}
\begin{proof}
For any $f \in \langle F \rangle$ with $\lm(H_{(u,v)}(f)) = h^{\lambda} x^{\alpha} \d^{\beta}$,
there exists $g^{(h)} \in G^{(h)}$ with $\lm(g^{(h)}) = h^{\kappa} x^{\gamma} \d^{\delta}$
satisfying $\lm(g^{(h)}) \mid \lm(f)$.
Then $\lm(g^{(h)})_{\mid_{h=1}} = x^{\gamma} \d^{\delta} \mid x^{\alpha} \d^{\beta} = \lm(f)_{\mid_{h=1}}$,
which proves the claim.
\end{proof}

\begin{theorem} \label{Theorem: initial ideal}
Let $\prec$ be a global monomial ordering on $D$ and $\prec_{(-w,w)}$ the non-global ordering defined by
\begin{alignat*}{2}
		x^{\alpha} \d^{\beta} \prec_{(-w,w)} x^{\gamma} \d^{\delta}
		& \quad\text{if}\quad &-w \alpha + w \beta &< -w \gamma + w \delta\\
		& \quad\text{or}\quad &-w \alpha + w \beta &= -w \gamma + w \delta
			\quad\text{and}\quad x^{\alpha} \d^{\beta} \prec x^{\gamma} \d^{\delta}.
\end{alignat*}
If $G^{(h)}$ is a Gr\"obner basis of $H_{(u,v)}(I)$ with respect to $\prec^{(h)}_{(-w,w)}$,
then $\{ \inwh(g) \mid g \in G^{(h)}\}$ is a Gr\"obner basis of $\inwh(H_{(u,v)}(I))$ with respect to $\prec^{(h)}$.
\end{theorem}
\begin{proof}
Let $f' \in \inwh(H_{(u,v)}(I))$ be $(-w,w,0)$-ho\-mo\-ge\-neous. 
There exist elements $f \in H_{(u,v)}(I), g \in G^{(h)}$ such that 
$f' = \inwh(f)$ and $\lm_{\prec^{(h)}_{(-w,w)}}(g) \mid \lm_{\prec^{(h)}_{(-w,w)}}(f)$.
Since $f,g$ are $(u,v)$-ho\-mo\-ge\-neous, we have
\[
\lm_{\prec^{(h)}_{(-w,w)}}(p) = \lm_{\prec^{(h)}}(\inwh(p)) \text{ for } p \in \{ f,g \},
\]
which finishes
the proof.
\end{proof}

Summarizing the results from this section, we obtain the following algorithm to compute the initial ideal.

\begin{algorithm}[\texttt{InitialIdeal}] \label{InitialIdeal} ~
\begin{algorithmic}
\STATE 
	\REQUIRE $I \subset D$ such that $D/I$ is holonomic,
	 $\prec$ a global ordering on $D$, 
	 $0 \neq w \in \R^n_{\geq 0}$,
	 $u,v \in \R^n_{>0}$
	\ENSURE A Gr\"obner basis $G$ of $\inw(I)$ with respect to $\prec$
	\STATE $\prec^{(h)}_{(-w,w)} :=$ the homogenized ordering as in Theorem \ref{Theorem: initial ideal}
	\STATE $G^{(h)} :=$ a Gr\"obner basis of $H_{(u,v)}(I)$ with respect to $\prec^{(h)}_{(-w,w)}$
	\RETURN $G = \inw({G^{(h)}}_{\mid_{h=1}})$
\end{algorithmic}
\end{algorithm}

\subsection{Intersecting an ideal with a principal subalgebra} \label{intersection}

We will now consider a much more general setting than needed to compute the global $b$-function.
Let $A$ be an associative $\K$-algebra. 
We are interested in computing the intersection of a left ideal $J \subset A$ with the subalgebra $\K[s]$ of $A$ where $s \in A \setminus \K$. 
We would like
to find the monic 
polynomial $b \in A$ such that
\[
	\langle b \rangle = J \cap \K[s].
\]

For this section, we will assume that there is a monomial ordering on $A$ such 
$J$ has a finite left Gr\"obner basis $G$.

Then we can distinguish between the following four situations:
\begin{enumerate}[label=\arabic{*}., ref=\arabic{*}.] 
	\item \label{sit 1} No leading monomials of elements in $G$ divide the leading monomial of any power of $s$.
	\item \label{sit 2} There is an element in $G$ whose leading monomial divides the leading monomial of some power of $s$.
		In this situation, we have the following sub-situations.
		\begin{enumerate}[label=2.\arabic{*}., ref=2.\arabic{*}.]
			\item \label{sit 2.1} $J \cdot s \subset J$ and $\dim_{\K}(\mathrm{End}_{A}(A/J)) < \infty$.
			\item \label{sit 2.2} One of the two conditions in 2.1. does not hold.
			\begin{enumerate}[label=2.2.\arabic{*}., ref=2.2.\arabic{*}.]
				\item \label{sit 2.2.1} The intersection is zero.
				\item \label{sit 2.2.2} The intersection is not zero.
			\end{enumerate}
		\end{enumerate}
\end{enumerate}

We now consider the first case.

\begin{lemma} \label{la:zero intersection}
If there exists no $g \in G$ such that $\lm(g)$ divides $\lm(s^k)$
for some $k \in \N_0$, 
then $J \cap \K[s] = \{ 0 \}$.
\end{lemma}
\begin{proof}
Let $0 \neq b \in J \cap \K[s]$.
Then $\lm(b) = \lm(s^k)$ for some $k \in \N_0$.
Since $b \in J$, there exists $g \in G$ such that $\lm(g) \mid \lm(b) = \lm(s^k)$.
\end{proof}

In the second situation however, we cannot in general state whether the intersection is trivial or not as the following example illustrates.

\begin{remark}
The converse of the previous lemma does not hold. For instance, consider $\K[x,y]$ and $J = \langle y^2 + x \rangle$.
Then $J \cap \K[y] = \{ 0 \}$ while $\{ y^2 + x\}$ is a Gr\"obner basis of $J$ for any ordering.
\end{remark}

In situation \ref{sit 2.1} though, the intersection is not zero as the following lemma shows, inspired by the sketch of the proof of Theorem \ref{b(s) is not zero} in \cite{SST00}.

\begin{lemma} \label{lemma minimal polynomial}
Let $J \cdot s \subset J$ and $\dim_{\K}(\mathrm{End}_{A}(A/J)) < \infty$. Then $J \cap \K[s] \neq \{ 0 \}$.
\end{lemma}
\begin{proof}
Consider the right multiplication with $s$ as a map $A/J \ra A/J$ which is a well-defined $A$-module endomorphism of $A/J$ as $a-a' \in J$ implies that $(a-a')s \in J \cdot s \subset J$, which holds by assumption for all $a,a' \in A$.
Since $\mathrm{End}_{A}(A/J)$ is finite dimensional, linear algebra guarantees that this endomorphism has a well-defined non-zero minimal polynomial $\mu$.
Moreover, $\mu$ is precisely the monic generator of $J \cap \K[s]$ as $\mu(s) = [0]$ in $A/J$, hence $\mu(s) \in J \cap \K[s]$, and $\deg(\mu)$ is minimal by definition.
\end{proof}

\begin{remark} \label{End(A/J)}
In particular, the lemma holds if $A/J$ itself is a finite dimensional $A$-module.
In the case where $A$ is a Weyl algebra and $A/J$ is a holonomic module, we know that 
$\dim_{\K}(\mathrm{End}_{A}(A/J)) < \infty$ holds (e.~g. \cite{SST00}).
\end{remark}

By the proof of the lemma, we have reduced our problem of intersecting an ideal with a subalgebra generated by one element to a problem from linear algebra, namely to the one of finding the minimal polynomial of an endomorphism.

\begin{proof}[Proof of Theorem \ref{b(s) is not zero}]
Let $0 \neq w \in \R^n_{\geq 0}, I \subset D$ such that $D/I$ is a holonomic module, $J := \inw(I)$ and $s := \sum_{i=1}^n w_i x_i \d_i$.
Without loss of generality let $0 \neq p = \sum_{\alpha,\beta} c_{\alpha,\beta} x^{\alpha} \d^{\beta} \in J$ be $(-w,w)$-homogeneous.
Then we obtain for every monomial in $p$ by using the Leibniz rule
\begin{align*}\begin{split}
&x^{\alpha} \d^{\beta} x_i \d_i
= x^{\alpha+e_i} \d^{\beta+e_i} + \beta_i x^{\alpha} \d^{\beta}\\
&\quad= (\d_i x_i^{\alpha_i+1} - (\alpha_i + 1) x_i^{\alpha_i}) \frac{x^{\alpha}}{x_i^{\alpha_i}} \d^{\beta} 
	+ \beta_i x^{\alpha} \d^{\beta}\\
&\quad= (\d_i x_i - (\alpha_i + 1) + \beta_i) x^{\alpha} \d^{\beta}
= (x_i \d_i - \alpha_i + \beta_i) x^{\alpha} \d^{\beta}.
\end{split}\end{align*}
Put $m = -w \alpha + w \beta$ for some term $c_{\alpha,\beta} x^{\alpha} \d^{\beta}$ in $p$ where $c_{\alpha,\beta}$ is non-zero. Since $p$ is $(-w,w)$-homogeneous, $m$ does not depend on the choice of this term.
Hence,
\begin{align*}\begin{split}
&p \cdot s
= p \sum_{i=1}^n w_i x_i \d_i
= \sum_{i=1}^n w_i \sum_{\alpha,\beta} (x_i \d_i - \alpha_i + \beta_i) c_{\alpha,\beta} x^{\alpha} \d^{\beta} \\
&\, = s \cdot p + \sum_{i=1}^n \sum_{\alpha,\beta} w_i (- \alpha_i + \beta_i) c_{\alpha,\beta} x^{\alpha} \d^{\beta} 
 = (s + m) \cdot p \in J.
\end{split}\end{align*}
Therefore, $J \cdot s \subset J$ holds.
Since $D/J$ is holonomic \citep{SST00}, Remark \ref{End(A/J)} and Lemma \ref{lemma minimal polynomial} yield the claim.
\end{proof}

If one knows in advance that the intersection is not zero, the following algorithm can be used for computing.

\begin{algorithm}[\texttt{PrincipalIntersect}] \label{PrincipalIntersect}
\begin{algorithmic}
  \STATE 
	\REQUIRE $s \in A, J \subset A$ a left ideal such that $J \cap \K[s] \neq \{ 0 \}$.
	\ENSURE $b \in \K[s]$ monic such that $J \cap \K[s] = \langle b \rangle$
	\STATE $G :=$ a finite left Gr\"obner basis of $J$ (assume it exists)
	\STATE $i := 1$
	\LOOP
		\IF{there exist $a_0, \ldots, a_{i-1} \in \K$ such that\\
			\medskip
			$\phantom{if there }\NF(s^i,G) + \sum_{j=0}^{i-1} a_j \NF(s^j,G) = 0$}
			\medskip
  	 	\RETURN $b := s^i + \sum_{j=0}^{i-1} a_j s^j$  	
  	 \ELSE \STATE $i := i+1$
 		\ENDIF
	\ENDLOOP
\end{algorithmic}
\end{algorithm}

Note that because $\NF(s^i,G) + \sum_{j=0}^{i-1} a_j \NF(s^j,G) = 0$ is equivalent to $s^i + \sum_{j=0}^{i-1} a_j s^j \in J$, 
the algorithm searches for a monic polynomial in $\K[s]$ that also lies in $J$. This is done by going degree by degree through the powers of $s$ until there is a linear dependency. This approach also ensures the minimality of the degree of the output.
The algorithm terminates if and only if $J \cap \K[s] \neq \{ 0 \}$.
Note that this approach works over any field.

%
%
The check whether there is a linear dependency over $\K$ between the computed normal forms of the powers of $s$ can be done by means of linear algebra.

\subsubsection{Applications} \label{PIapps}

Apart from computing global $b$-functions, there are various other applications of Algorithm \ref{PrincipalIntersect}.

\paragraph{Solving zero-dimensional systems}

Recall that an ideal $I \subset \K[x_1,\ldots,x_n]$ is called ze\-ro-di\-men\-sional if $\K[x_1,\ldots,x_n] / I$ is finite dimensional as a $\K$-vector space. It is known (e.~g. by Lemma \ref{lemma minimal polynomial}) that in this case there exist $0 \neq f_i \in I \cap \K[x_i]$ for each $1 \leq i \leq n$, which implies that the cardinality of the zero-set of $I$ is finite.

In order to compute this zero-set, one can use the classical triangularization algorithms.
These algorithms require to compute a Gr\"obner basis with respect to some elimination ordering (like lexicographic one), which might be very hard.

By Algorithm \ref{PrincipalIntersect}, a generator of $I \cap \K[x_i]$ can be computed without these expensive orderings. Instead, any ordering, hence a better suited one, may be freely chosen.

A similar approach is used in the celebrated FGLM algorithm \citep{FGLM93}. See also \cite{NK99} for a different approach.

\paragraph{Computing central characters}

Let $A$ be an associative $\K$-algebra. The intersection of a left ideal with the center of $A$, which is isomorphic to a commutative polynomial ring, is important for many algorithms, among other for the computation of the central character decomposition of a finitely presented module (cf. \cite{Lev05}).
In the situation, where the center of $A$ is generated by one element (which is not seldom), we can apply Algorithm \ref{PrincipalIntersect} to compute the intersection (known to be often quite nontrivial) without engaging much more expensive Gr\"obner basis computations, which use elimination.

\begin{example}
Consider the universal enveloping algebra of the
Lie algebra $\mathfrak{sl}_2$, $A = U(\mathfrak{sl}_2,\K) = 
\K \langle  e,f,h \mid [e,f]=h,\; [h,e]=2e,\;[h,f]=-2f \rangle$. 
It is known, that over a field of characteristic 0, the center of $A$ is
$\K[4ef+h^2-2h]$. Consider a left  ideal $L$ and a two-sided ideal $T$, both generated
by $G = \{ e^{11}, f^{12}, h^5-10h^3+9h \}\subset A$. Then consider $A$-modules 
$M_L = A/L$ and $M_T = A/T$, which turn out to be finite-dimensional over $\K$. 
We are interested in intersecting $L,T$ with $Z(A)$ and factorizing the output polynomial in one variable.
The implementation of the Algorithm \ref{PrincipalIntersect} in the library \texttt{bfun.lib} is described in Section \ref{bfct-impl}.

\begin{verbatim}
LIB "ncalg.lib"; LIB "central.lib"; LIB "bfun.lib";
def A = makeUsl(2); setring A;   // U(sl_2,Q)
ideal Z = center(2);             // generators of deg <= 2
poly z = Z[1];                   // we know there is just 1 generator
ideal I = e^11,f^12,(h-3)*(h-1)*h*(h+1)*(h+3);
ideal L = std(I);                // left GB of I
vdim(L);                         // K-dimension of A/I
==> 559
vector vL = pIntersect(z,L);     // L \cap K[z]
ideal T = twostd(I);             // twosided GB of I
vdim(T);                         // K-dimension of A/T
==> 21
vector vT = pIntersect(z,T);     // T \cap K[z]
ring r = 0,z,dp;                 // commutative univariate ring
// pretty-print factorization of polynomials:
print(matrix(factorize(vec2poly(imap(A,vT)),1)));
==> z-3,z,z-15
print(matrix(factorize(vec2poly(imap(A,vL)),1)));
==> z-3,z,z-440,z-8,z-48,z-168,z-15,z-99,z-120,
    z-255,z-483,z-575,z+1,z-399,z-143,z-195,z-63,
    z-80,z-288,z-360,z-224,z-323,z-35,z-24
\end{verbatim}

Note, that all the computations, thanks to Algorithm \ref{PrincipalIntersect}, were completed
in a couple of seconds, while the Gr\"obner-driven approach was still running after 20 minutes.
\end{example}

\subsection{Intersecting an ideal with a multivariate subalgebra} 
\label{Imulti}

We now consider the case where we intersect $J$ with the subalgebra $\K[s] = \K[s_1,\ldots,s_r]$ of an associative $\K$-algebra $A$ for nonconstant, pairwise commuting $s_1,\ldots,s_r \in A$.

The following result is a consequence of a well-known characterization of zero-di\-men\-sion\-al ideals.
\begin{lemma}
The ideal $J \cap \K[s]$ is zero-dimensional if and only if for all $1 \leq i \leq r$ there exist $f_i \in J$ such that $\lm(f_i) = s_i^{d_i}$ for some $d_i \in \N_0$.
\end{lemma}

\begin{lemma} For a finite left Gr\"obner basis $G$ of $J$,
\begin{align*}
	\GKdim(\K[s]) 
	&\geq \GKdim(\K[s] / (J \cap \K[s]))\\
	&\geq \GKdim(\K[s] / (L(G) \cap \K[s])).
\end{align*}
\end{lemma}
\begin{proof}
For all $f \in J \cap \K[s]$ there exists $g \in G$ such that $\lm(g) \mid \lm(f)$, 
which implies $\lm(g) \in \K[s]$ and thus, the claim follows.
\end{proof}
Note that the first inequality is strict if and only if $J \cap \K[s] \neq \{ 0 \}$.

We give a generalization of Algorithm \ref{PrincipalIntersect} to compute a partial Gr\"obner basis of $J \cap \K[s]$ up to a specified bound $k \in \N$.

\begin{algorithm}[\texttt{IntersectUpTo}] \label{IntersectUpTo}
\begin{algorithmic}
	\STATE 
	\REQUIRE $s_1,\ldots,s_r \in A$ pairwise commuting, 
		$J \subset A$ a left ideal, $k \in \N$ an upper degree bound
	\ENSURE a GB for $J \cap \K[s_1,\ldots,s_r]$ up to degree $k$
	\STATE $G :=$ a partial left Gr\"obner basis of $J$ consisting of elements up to degree $k$
	\STATE $d := 0$
	\STATE $B := \emptyset$
	\WHILE{$d \leq k$}
		\STATE $M_d := \{ s^{\alpha} \mid |\alpha| \leq d \}$
		\IF{there exist $a_m \in \K$, not all $0$, such that
			$\sum\limits_{m \in M_d} a_m \NF(m,G) = 0$}
			\IF{$\sum\limits_{m \in M_d} a_m m \notin \langle B \rangle$}
  	 		\STATE $B := B \cup \{\sum\limits_{m \in M_d} a_m m \}$
  	 	\ENDIF
 		\ENDIF
		\STATE $d := d+1$
	\ENDWHILE
	\RETURN $B$
\end{algorithmic}
\end{algorithm}

A couple of improvements can be made to speed up the computation time.

If $p \in B$ with $\lm(p) = m$ has been found, any monomial which is a multiple of $m$ 
can be discarded in the following iterations.

Let $G$ be a Gr\"obner basis of $J$ with respect to some fixed ordering $\prec$. 
By using $p \in J \cap \K[s]$ if and only if $\lm(p) \in L(G) \cap \K[s]$, 
one may disregard 
$\{ m \in M_d \mid \max_{\prec}(m' \in L(G) \cap M_d) \prec m \}$.

Further note that $\NF(m,G) = m$, if $m \notin L(G) \cap \K[s]$.

Using these improvements and choosing $\prec$ to be a degree ordering and the elements in $B$ to be monic, the output of the algorithm equals the reduced Gr\"obner basis of $J \cap \K[s]$ with respect to $\prec$ up to degree $k$. However, in general no termination criterion is known to us yet, that is apriori we do not know when we already have the complete needed basis of the intersection. Nevertheless, 
%
the termination is predictable if $J \cap \K[s]$ is a principal ideal in $\K[s]$. This situation often arises in the computation of \BS ideals, see Section \ref{BSideal}. Moreover, another possibility for the algorithm to stop will be when the set of monomials we consider becomes empty on some step, which is the case if and only if $J \cap \K[s]$ is zero-dimensional. 

As one can see, the results above can be generalized by replacing the commutativity condition for a subalgebra $S$ with the condition, that $S$ is a $G$-algebra in a $\K$-algebra $A$. This and further generalizations will be studied in the next articles. Note, that under some extra requirements the algorithm will terminate after finally many steps without setting an explicit degree bound. Hence, in such cases a generally complicated elimination with Gr\"obner bases can be replaced by much easier and predictable Gr\"obner-free approach. The latter will, of course, allow to solve harder computational problems.

As it was noted in \cite{LVint}, even the existence of a certain elimination ordering in $G$-algebras is not guaranteed. Consider the algebra $B = \K\langle x,y \mid yx=xy+y^2 \rangle$. Then the ordering condition of Def. \ref{GalgLie} says $x>y$ must hold for any ordering. Hence, we cannot use Gr\"obner basis in this $G$-algebra for computing the intersection of an ideal $I$ with
the subalgebra $\K[x]$, since the latter requires the use of ordering with $x<y$. One possibility
would be to consider $B$ as a $\K$-algebra with the ordering $x<y$ modulo the 
two-sided ideal, generated by $y^2 - yx + xy$. But this ideal has infinite two-sided Gr\"obner basis, hence doing the elimination via passing to $\K$-algebra setting is problematic, since it depends on the input ideal $I$.

Despite these complications, it is obvious, that the preimage of an ideal in a subalgebra does exist. Hence, Algorithm \ref{IntersectUpTo} is indeed the only computational possibility to get some information about such a preimage.

\section{Bernstein-Sato Polynomial of $f$} \label{globalBS}

Let $f \in \K[x_1,\ldots,x_n]$. One possibility to define the Bernstein-Sato polynomial of $f$ is to apply the global $b$-function for specific weights.

\begin{definition}
Let $B(s)$ denote the global $b$-function of the univariate Malgrange ideal $I_f$ of $f$ (cf. Section \ref{sAnn}) with respect to the weight vector $w = (1,0,\ldots,0) \in \R^{n+1}$, that is the weight of $\d t$ is $1$. Then $b(s) := B(-s-1)$ is called the \emph{global $b$-function} or the \emph{Bernstein-Sato polynomial} of $f$.
\end{definition}

By Theorem \ref{b(s) is not zero}, $b(s) \neq 0$ holds. Moreover, it is well known that $-1$ is always a root  of the Bernstein-Sato polynomial for nonconstant $f$ and Kashiwara proved that all its roots are negative rational numbers \citep{Kashiwara76/77}.

The following version of Bernstein's theorem \citep{Bernstein71} gives us another option to define the Bernstein-Sato polynomial.

\begin{theorem}[\cite{SST00}] \label{bfct and ann}
The Bernstein-Sato poly\-no\-mi\-al $b(s)$ of $f$ is the unique monic polynomial of minimal degree in $\K[s]$ satisfying the identity
\[
	P \bullet f^{s+1} = b(s) \cdot f^s \qquad \text{for some operator } P \in D[s].
\]
\end{theorem}

Since $P \cdot f - b(s) \in \Ann(f^s)$ holds, the Bernstein-Sato polynomial is the monic polynomial of minimal degree in $\K[s]$ that also lies in $\Ann(f^s) + \langle f \rangle$, hence $b(s)$ is the monic generator of this intersection.

\begin{remark}
\label{2methods}
There are the following choices for computing the Bernstein-Sato polynomial:
\begin{enumerate}
	\item Compute a Gr\"obner basis either of \\
		(a) $J = \ini_{(-w,w)}(I_f)$, which amounts to 1 Gr\"obner basis computation in $D_n$ or\\ 
		(b) $J = \ann(f^s) + \langle f \rangle$, which requires 2 Gr\"obner basis computations in $D_n[s]$.
	\item Intersect $J$ with $\K[\xi]$ where $\xi$ is either $s$ or $\sum_i w_i x_i \d_i$\\
		(a) the classical elimination-driven approach (needs 1 tough Gr\"obner basis computation) or\\
		(b) using Algorithm \ref{PrincipalIntersect} with no Gr\"obner basis computation. 
\end{enumerate}
It is very interesting to investigate the new approach for the computation
of \BS polynomials, arising as the combination of the two methods
\begin{enumerate}
	\item $\Ann(f^s)$ via Brian\c{c}on-Maisonobe (cf. Section \ref{sAnn} and \cite{LM08}),
	\item $(\Ann(f^s) + \langle f \rangle) \cap \K[s]$ via Algorithm \ref{PrincipalIntersect}.
\end{enumerate}
\end{remark}

For the computation of $\ini_{(-w,w)}(I_f)$ using the method of weighted homogenization as described in Section \ref{initial ideal}, the following choice of weights is proposed in \cite{Noro02} for an efficient Gr\"obner basis computation:
\begin{alignat*}{1}
	u &= (\deg_{\hat{u}}(f), \hat{u}_1, \ldots, \hat{u}_n ),\\
	v &= (1, \deg_{\hat{u}}(f) - \hat{u}_1 + 1, \ldots, \deg_{\hat{u}}(f) - \hat{u}_n + 1)
\end{alignat*}
such that the weight of $t$ is $\deg_{\hat{u}}(f)$ and the weight of $\d_t$ is $1$.
Here, $\hat{u} \in \R^n_{>0}$ is an arbitrary vector and $\deg_{\hat{u}}(f)$ denotes the weighted total degree of $f$ with respect to $\hat{u}$.
The vector $\hat{u}$ may be choosen heuristically in accordance to the shape of $f$ or by default, one can set $\hat{u} = (1,\ldots,1)$.

\section{Enhancements to steps of algorithms} \label{enhancements}

\subsection{Enhanced computation of $\Ann_{D[s]} (f^s)$}

Consider the set of generators $G:= \{ f \d t + s, \{ f_i \d t + \d_i \mid 1\leq i \leq n \}$ of an ideal $J$, coming from the \BM method. According
to the latter, we have to eliminate $\d t$ from $J$, that is to compute $J \cap D_n[s] = \Ann_{D[s]} (f^s)$.

Since any element $h$ from $J$ has a presentation as
\[
h = a_0 (f \d t+ s) + \sum_{i=1}^n a_i( f_i \d t + \d_i) = 
(a_0 f + \sum_{i=1}^n a_i f_i) \d t + (a_0 s + \sum_{i=1}^n a_i \d_i),
\] 

then for all $(a_0,a_1,\ldots,a_n) \in \Syz(\{f, f_1,\ldots,f_n \}) \cap \K[x,s]^{n+1}$
we obtain that $a_0 s + \sum_{i=1}^n a_i \d_i \in J \cap D_n[s]$.

Moreover, it is known, that indeed the
above elements generate the $\K[x]$-submodule of all the elements in $J \cap D_n[s]$, which total degree in $\d_i$ does not exceed 1.

Consider the set $T_f=\{ f, f_1,\ldots,f_n \} \subset \K[s] \subset D_n[s]$, a left ideal $D_n[s] T_f \subset D_n[s]$ and an ideal $\K[x] T_f \subset \K[x]$. Then a Gr\"obner basis of $\K[x] T_f$ is a Gr\"obner basis for $D_n[s] T_f$ as well. Denote by 
$S_f$ a set of generators of the module $\Syz(T_f) \subset \K[x]^{n+1}$.
By e.~g. generalized Schreyer's theorem \citep{LVdiss}, it follows that
the module of left syzygies $\LeftSyz_{D_n[s]}(T_f) = D_n[s] S_f$.

Let $\prec_1$ be a monomial module ordering on $\K[x]^{n+1}$, which is a position-over-term ordering, which gives preference to the 1st component. Since degree of $f$ is always by 1 bigger than the degree of $\tfrac{\d f}{\d x_i}$, the cofactors to $f$ have respectively smaller degree. 

\begin{algorithm}[\texttt{SannfsBMSyz}] \label{SannfsBMSyz}~
\begin{algorithmic}
	\REQUIRE $f\in \K[x]$
	\ENSURE $\Ann_{D_n[s]} (f^s)$
	\STATE $T_f := \{ f, \frac{\d f}{\d x_1},\ldots, \frac{\d f}{\d x_n} \} \subset \K[x]$
	\STATE $S_f := \Syz(T_f) \subset \K[x]^{n+1}$ \\
	\STATE $S_f :=$ \textsc{Gr\"obnerBasis}$(S_f)$ with respect to $\prec_1$
	\STATE create ring $D_n[s]$
	\STATE form $S_a := \{a_0 s + \sum_{i=1}^n a_i \d_i \}$ for every gen $a$ of $S_f$
	\STATE $S_a := $ \textsc{Gr\"obnerBasis}$(S_a)\in D[s]$ with respect to an ordering
	\STATE $G := \{ f \d t + s, \frac{\d f}{\d x_1} {\d t} + \d_1, \ldots, \frac{\d f}{\d x_n} {\d t} + \d_n \}\subset D \langle \d t, s \rangle$
	\STATE $G := $\textsc{Gr\"obnerBasis}$(G \cup S_a)$ with respect to an elimination ordering for $\d t$
	\RETURN $(G \cap D[s])$
\end{algorithmic}
\end{algorithm}  

\begin{remark}
One of major difficulties in the computation of Gr\"obner basis (especially with respect to an elimination ordering) is the need to compute numerous intermediate polynomials (of usually high degree and with big coefficients) in order to
come to a polynomial in the answer, which is often of small degree with coefficients of moderate size. Actually the set of generators $S_a$, which we compute in the syzygy-driven algorithm, generates already a part of the answer, though the corresponding ideal is, in general, not yet the complete answer. 

Computing a Gr\"obner basis of $S_a$ and adding it to the original set of generators $G$ allows to avoid at first place the discovery of elements of $S_a$ in the Gr\"obner basis computation of $G\cup S_a$ and hence allows to decrease the number of intermediate unpleasant polynomials, which are needed in such computation. This is important, since in the answer there are no polynomials of degree zero with respect to $\d_i$, that is $\Ann_{D_n[s]} (f^s) \cap \K[x,s]=0$ due to the fact, that the only element from the ring $\K[x,s]$, annihilating $f^s$, is zero. Hence with $S_a$ we add the set of elements of  smallest possible total degree in $\d_i$ that is of degree 1. Such elements are, in general, very hard to compute via the Gr\"obner-driven elimination. 

However, it is very interesting to derive conditions, under which the above algorithm is more efficient than the one of Brian\c{c}on-Maisonobe. We observe that it is not true for a couple of examples. See section \ref{compare}.
\end{remark}

\subsection{Enhanced computation of $b_f(s)$}


In the following Lemma we collect folklore results and supply them with short proofs
for the completeness of exposition.

\begin{lemma} \label{smallThings}
Let, as before, $f\in\K[x]\setminus\{0\}$.
\begin{enumerate}
\item \label{ratAnn} $\forall \ 1\leq i \leq n$ we have $f \d_i - s \frac{\d f}{\d x_i} \in \Ann_{D[s]} f^s$ and $ \frac{\d f}{\d x_i} \d_j -  \frac{\d f}{\d x_j} \d_i \in \Ann_{D[s]} f^s$
\item \label{minint}  For $f\in\K$, $b_f(s)=1$. For $f\in\K[x]\setminus\{\K\}$,  $(s+1) \mid b_f(s)$.
\item \label{reducedB}
\[
\langle \dfrac{b(s)}{s+1} \rangle = (\Ann f^s + \langle f, \dfrac{\d f}{\d x_1}, \ldots, \dfrac{\d f}{\d x_n} \rangle) \cap \K[s]
\]
\item \label{smooth} if $\K=\bar{\K}$, then $V(f)$ is smooth implies $b_f(s)=s+1$.
\end{enumerate}
\end{lemma}
\begin{proof}
We use shortcut $f_i := \dfrac{\d f}{\d x_i}$. Consider surjective $\K$-alg homomorphism
$\pi_{\alpha}: D[s] \to D$, $s \mapsto \alpha$ and apply it to the inclusion
${}_{D[s]}\langle P(s) f - b(s) \rangle \subset \Ann_{D[s]} f^{s}$.
Then we have an inclusion 
\[
\pi_{\alpha}(\langle P(s) f - b(s) \rangle) =
\langle P(\alpha) f - b(\alpha) \rangle \subset (\Ann_{D[s]} f^{s})\mid_{s=\alpha} \subseteq
\Ann_{D} f^{\alpha}
\]
\begin{enumerate}
\item 
Direct calculation. 
\item 
By using $\pi_{-1}$ from above, we obtain $P(-1) = b(-1) f^{-1}$ modulo $\Ann_D f^{0}  =$  $\langle \d_1,\ldots, \d_n \rangle$. Hence $P(-1) = p(x) \in \K[x]$ 
and $p(x) = b(-1) f^{-1}$, which can be true only in two cases:\\
1. $f\in \K^*$, then $P(s) = f^{-1} \in \K$ and $b(s)=1$, \\
2. $f\not\in \K^*$, then $b(-1) = 0$.
\item 
Let us write $P(s) = \sum_i P_i \d_i + P_0$ for $P_0\in\K[x,s]$ and $P_i \in D[s]$. 
Computing modulo $\Ann_{D[s]} f^s$ and using (\ref{ratAnn})
we can present $P(s)f = P_0 f + \sum_i P_i \d_i f = P_0 f + (s+1)\sum_i P_i f_i$. 
By (\ref{minint}) $b(-1)=0$, hence by specializing $s$ to $-1$ in Bernstein's equation we get
$P(-1)\bullet 1 = b(-1)f^{-1} = 0$. Thus $P(-1) \in \Ann_{D[s]} (1) = \langle \d_1, \ldots, \d_n \rangle$. In particular, $P_0(-1) = 0$ and hence $s+1 \mid P_0 \in\K[x,s]$. 
Moreover,
\[
\sum_i P_i f_i + \tfrac{P_0}{s+1} f - \tfrac{b(s)}{s+1} \in \Ann_{D[s]} f^s
\]
and the claim follows.
\item Since $I = I(Sing(V(f))) = \langle f,  f_1, \ldots, f_n \rangle \subset \K[x]$, smoothness takes place when $1\in I$, hence by (\ref{reducedB}) we have $1 \in \langle \frac{b(s)}{s+1} \rangle$ 
and thus, $b(s) = s+1$.
\qedhere
\end{enumerate}
\end{proof}

\begin{lemma} \label{compareReducedB}
For a fixed algorithm, which computes $\Ann_{D[s]} (f^s)$, let us consider two ideals $I_1 = \Ann_{D[s]} (f^s) + \langle f \rangle$ and $I_2 = \Ann_{D[s]} (f^s) + \langle f, \frac{\d f}{\d x_1}, \ldots, \frac{\d f}{\d x_n} \rangle$ (note that $I_1, I_2 \subset D[s])$. There are two following algorithms, which take an ideal and a polynomial as input and return Bernstein-Sato polynomial, namely 
\begin{itemize}
\item[] \textit{Algorithm 1. } $b(s) =$ \textsc{pIntersect}$(I_1,s)$.
\item[] \textit{Algorithm 2. } $b(s) = (s+1) \cdot$ \textsc{pIntersect}$(I_2,s)$.
\end{itemize}
Then the Algorithm 2 is more efficient
than the Algorithm 1.
\end{lemma}
\begin{proof}
Performing the principal intersection, the Algorithm 2 will compute one
normal form less (of an element of high degree) than the Algorithm 1. Moreover, the normal forms in Algorithm 2 are taken with respect to a bigger ideal, what makes respective computations easier as well.
By Lemma \ref{smallThings} (\ref{ratAnn}) we know that $f \d_i - s f_i \in \Ann_{D[s]} (f^s)$. Hence these elements can be reduced to $(s+1)\cdot f_i$ in $I_1 = \Ann_{D[s]} (f^s) + \langle f \rangle$. Meanwhile in $I_2 = \Ann_{D[s]} (f^s) + \langle f, f_1,\ldots,f_n \rangle$ we reduce 
$f \d_i - s f_i$ automatically to zero. Note, that indeed $(s+1) I_2 \subset I_1 \subset I_2$ holds and hence, in the process of computing a Gr\"obner basis of $I_1$ (Algorithm 1), the operations with commutative elements of the kind $(s+1)f_i$ will in general keep the factor $(s+1)$, thus operating with larger polynomials of higher degree. Hence the claim. 
\end{proof}


\subsection{Enhanced computation of normal forms}

When computing normal forms of the form $\NF(s^i,J)$ like in Algorithm \ref{PrincipalIntersect} we can speed up the reduction process by making use of the previously computed normal forms.

\begin{lemma} \label{NF computing 1}
Let $A$ be a $\K$-algebra, $J \subset A$ a left ideal and let $f \in A$.
For $i \in \N$ put $r_i = \NF(f^i,J)$, $q_i = f^i - r_i \in J$ and $c_i = \frac{\lc(q_i r_1)}{\lc(r_1 q_i)}$ provided $r_1 q_i \not=0$. For $r_1 q_i=0$ we put $c_i=0$.
Then we have for all $i \in \N$
\[
	r_{i+1} 
	= \NF(f r_i,J)
	= \NF([f^i-r_i,r_1]_{c_i} + r_i r_1,J).
\]
\end{lemma}
\begin{proof}
It holds that $f^{i+1} = f q_i + f r_i \ra f r_i$, which shows the first equation.
On the other hand, 
$f^{i+1} = q_i f + r_i f = q_i (q_1 + r_1) + r_i (q_1 + r_1) = q_i q_1 + q_i r_1 + r_i q_1 + r_i r_1
\ra q_i r_1 + r_i r_1 = (f^i - r_i) r_1 + r_i r_1 
\ra [f^i - r_i, r_1]_{c_i} + r_i r_1$, which proves the second equation.
\end{proof}

As a direct consequence, we obtain the following result for some $\K$-algebras of special importance.

\begin{corollary} \label{NF computing 2}
If $A$ is a $G$-algebra of Lie type (e.~g. a Weyl algebra),  then
\[
r_{i+1} = \NF(f r_i,J) = \NF([f^i-r_i,r_1] + r_i r_1,J) \text{  holds}.
\]
If $A$ is commutative, we have
\[
r_{i+1} = \NF(r_i r_1,J) = \NF(r_1^{i+1},J).
\]
\end{corollary}

Note, that computing Lie bracket $[f,g]$ both in theory and in practice is easier and faster, than to compute $[f,g]$ as $f\cdot g - g \cdot f$, see e.~g. \cite{LS03}.


\begin{remark}
We work on enhanced algorithms for the computation of the Bernstein operator from Theorem \ref{bfct and ann} as well and will report on the progress in forthcoming articles. 
\end{remark}


\section{Implementation}\label{bfct-impl}

In \cite{Noro02}, M. Noro proposed methods of modular change of ordering and 
modular solving of linear equations to be used in his approach, which is based on a kind of Algorithm \ref{PrincipalIntersect}.
In our implementation we decided to develop, test and enhance first 
purely characteristic $0$ methods, thus having the possibility to
adjoin modular methods later.

For the computation of $b$-functions and Bernstein-Sato polynomials, we offer the following procedures in the \textsc{Singular} library \texttt{bfun.lib}:

\texttt{bfct} 
computes $\inw(I_f)$ using weighted homogenization with weights $u,v$ for an optional weight vector $\hat{u}$ (by default $\hat{u} = (1,\ldots,1))$ as described above,
and then uses Algorithm \ref{PrincipalIntersect} with the enhancement from Corollary \ref{NF computing 2} for the intersection, where the occuring systems of linear equations are solved by means of linear algebra.

\texttt{bfctSyz}
computes $\inw(I_f)$ as in \texttt{bfct} and 
then uses Algorithm \ref{PrincipalIntersect}, where the linear equations are treated as polynomial ones and then solved by computing syzygies.

\texttt{bfctAnn}
computes $\Ann(f^s)$ via Algorithm \ref{SannfsBMSyz} and then computes the intersection of $\Ann(f^s) + \langle f, \frac{\d f}{\d x_1}, \ldots, \frac{\d f}{\d x_n} \rangle$ with $\K[s]$ analogously to \texttt{bfct}.

\texttt{bfctOneGB}
computes the initial ideal and the intersection at once using a homogenized elimination ordering (see also \cite{HH01}).

For the global $b$-function of an ideal $I$, \texttt{bfctIdeal} computes $\inw(I)$ using standard homogenization, i.~e.
weight\-ed homogenization where all weights are equal to $1$, and then proceeds the same way as \texttt{bfct}.
Recall that $D/I$ must be holonomic as in \cite{SST00}.

All these procedures work as the following example illustrates for \texttt{bfct} and the hyperplane arrangement  $xyz(y-z)(y+z)$.
\begin{verbatim}
LIB "bfun.lib";
ring r = 0,(x,y,z),dp;
poly f = x*y*z*(y-z)*(y+z);
bfct(f);
==> [1]:
==>    _[1]=-1
==>    _[2]=-5/4
==>    _[3]=-3/4
==>    _[4]=-3/2
==>    _[5]=-1/2
==> [2]:
==>    3,1,1,1,1
\end{verbatim}

\subsection{Comparison} \label{compare}

We use the polynomials in Table \ref{table:examples} for test examples, where 
we measure the total running time of each call to a system in a batch mode. In this
time the initialization of a system, loading of an example file, the actual computation
and the writing of an output are included.

\begin{table}[h t]
\begin{center}
\caption{Examples}
\begin{tabular}{|c|c|}
\hline
Example & Input \\
\hline
ab23 & $(z^2 + w^3)(2xz + 3yw^2)$ \\ 
cnu6 & $(xz+y)(x^6-y^6)$ \\ 
cnu7 & $(xz+y)(x^7-y^7)$ \\ 
tt43 & $x^4 + y^4 + z^4 – (xyz)^3$ \\ 
xyzcusp45 & $(xz+y)(x^4+y^5)$ \\ 
uw18 & $xyz(x-z)(z-y)(y+z)(2x+2y-z)$ \\ 
uw22 & $xyz(x+z)(x-y)(y-z)(y+z)$ \\ 
uw27 & $xyz (x+y) (-x+2y+z) (x+y+z) (y+z)$ \\ 
uw28 & $xyz (x+z) (-x+y+z) (x+y) (y+z)$ \\ 
uw29 & $xyz (x+z) (x+y) (-3x-y+2z) (y+z)$ \\ 
uw30 & $xyz (x-2z) (-x+y+z) (x-y) (y+z)$ \\ 
\hline
\end{tabular}
\label{table:examples}
\end{center}
\end{table}

The running times in the tables below are given in ``[hours[h]:]minutes:seconds'' format.
We use the shortcuts ${t}^{\times}$ when we have stopped the process after the time $t$ and ${t}^{\dagger}$ when the process ran out of memory after the time $t$. 

The tests were performed on a machine with 4 Dual Core AMD Opteron 64 Processor 8220 (2800 MHz) (only one processor could be used at a time) equipped with 32 GB RAM (at most 16 GB were allowed to us) running openSUSE 11 Linux. 

We first request the computation of $\Ann_{D_n[s]}(f^s)$ and the Bernstein-Sato polynomial comparing the different algorithms from Section \ref{enhancements}. We use the notation from Lemma \ref{compareReducedB}.

\begin{table}[h t]
\begin{center}
\caption{Comparison of the algorithms from Section \ref{enhancements}}
\begin{tabular}{|c||c|c||c|c|c|c|}
\hline
 & \multicolumn{2}{|c||}{$\Ann_{D_n[s]}(f^s)$} & \multicolumn{4}{|c|}{Bernstein-Sato polynomial} \\
 &  \multicolumn{2}{|c||}{}  & \multicolumn{2}{|c|}{\texttt{SannfsBMSyz} based} & \multicolumn{2}{|c|}{\texttt{SannfsBM} based}\\
Example & \texttt{SannfsBMSyz} & \texttt{SannfsBM} & Alg. 1 & Alg. 2 & Alg. 1 & Alg. 2 \\ 
\hline
ab23 & \textbf{0:01} & 0:02 & 0:05 & \textbf{0:03} & 0:06 & \textbf{0:03} \\ 
cnu6 & \textbf{0:01} & \textbf{0:01} & \textbf{0:01} & \textbf{0:01} & \textbf{0:01} & \textbf{0:01} \\ 
cnu7 & \textbf{0:09} & 0:12 & \textbf{0:11} & \textbf{0:11} & 0:19 & 0:12 \\ 
tt43 & \textbf{0:01} & \textbf{0:01} & 0:03 & \textbf{0:01} & 0:03 & \textbf{0:01} \\ 
xyzcusp45 & \textbf{0:56} & 1:22 & 1:30 & 1:16 & 1:20 & \textbf{1:10} \\ 
uw18 & 3:38 & \textbf{0:04} & 17:36 & 12:38 & 16:24 & \textbf{11:15} \\ 
uw22 & 2h:04:01$^{\dagger}$ & \textbf{0:04} & 2h:17:36$^{\dagger}$ & 2h:16:55$^{\dagger}$ & 2h:28:07$^{\dagger}$ & 2h:28:53$^{\dagger}$ \\ 
\hline
\end{tabular}
\end{center}
\label{table:compareAlg}
\end{table}

Further, we compare our implementations for the computation of the Bernstein-Sato polynomial with the existing ones in the computer algebra systems \textsc{Risa/Asir} and \textsc{Macaulay2}.

\begin{table}[h t]
\begin{center}
\caption{Comparison of different systems}
\begin{tabular}{|c|cc|c|cc|}
\hline
 & \multicolumn{ 2}{|c|}{\textsc{Asir}} & \textsc{Macaulay2} & \multicolumn{ 2}{|c|}{\textsc{Singular}} \\ 
Example & \texttt{bfunction} & \texttt{bfct} & \texttt{globalBFunction} & \texttt{bfct} & \texttt{bfctAnn} \\ 
\hline
ab23 & 0:23 & 0:17 & 0:27 & 0:17 & \textbf{0:04} \\ 
cnu6 & 1:39 & 0:54 & 14:03 & \textbf{0:01} & \textbf{0:01} \\ 
cnu7 & 7:32 & 4:46 & 4h:03:39$^{\times}$ & \textbf{0:06} & 0:20 \\ 
tt43 & 0:07 & 0:05 & 0:05 & 0:17 & \textbf{0:01} \\ 
xyzcusp45 & 1:52 & \textbf{1:10} & 4h:18:35 & 3:05 & 3:01 \\ 
uw18 & 7:22 & 29h:35:54$^{\times}$ & 4h:08:16$^{\times}$ & \textbf{6:21} & 12:27 \\ 
uw22 & \textbf{2:12} & 4h:04:05$^{\times}$ & 4h:01:43$^{\times}$ & 2:24 & 2h:42:02 \\ 
uw27 & \textbf{2:37} & 3h:05:14$^{\times}$ & 11h:45:18$^{\times}$ & 4:40 & 6h:55:35$^{\times}$ \\ 
uw28 & \textbf{1:36} & 10h:23:40$^{\times}$ & 3h:03:00$^{\times}$ & 3:10 & 3h:03:32$^{\times}$ \\ 
uw29 & \textbf{1:48} & 3h:51:14$^{\times}$ & 10h:23:42$^{\times}$ & 2:52 & 3h:01:30$^{\times}$ \\ 
uw30 & \textbf{1:58} & 5h:14:18$^{\times}$ & 3h:06:57$^{\times}$ & 3:09 & 3h:00:13$^{\times}$ \\ 
\hline
\end{tabular}
\label{table:compareImpl}
\end{center}
\end{table}

We have used \textsc{Risa/Asir} version 20071022,
\textsc{Macaulay2} version 1.1 with version 1.0 of \texttt{Dmodules.m2}
and \textsc{Singular 3-1-0} with \texttt{bfun.lib} version 1.13.

We would like to stress, that in our implementation of \texttt{bfun.lib} we have restricted ourselves to the use of characteristic zero methods, in order to see what can we achieve with them. The implementation of \textsc{Asir} by M.~Noro \citep{Noro02} uses the methods in prime characteristic, which can be applied to our implementation as well. However, the values in the table above indicate, that the difference in timings is not devastating for our cause.

As the timings in Table \ref{table:compareImpl} suggest, the approach via the initial ideal seems to be specially well suited for hyperplane arrangements, while it looks like that the performance of the annihilator based method is better for other kind of input (we took non-quasi\-ho\-mo\-ge\-neous singularities).
See \cite{Walther05} for details about generic arrangements. \newpage

\section{Bernstein-Sato Polynomial for a Variety}

In the paper of Budur, Musta{\c{t}}{\v{a}} and Saito \citep{BMS06},
using the theory of $V$-filtrations of Kashiwara \citep{Kashiwara83} and Malgrange \citep{Malgrange83},
the theory of the Bernstein-Sato polynomial of an arbitrary variety has been
developed. We present here the construction of the Bernstein-Sato polynomial
of an affine algebraic variety.

Given two positive integers $n$ and $r$, for the rest of this section we
fix the indices $i,j,k,l$ ranging between $1$ and $r$ 
and an index $m$ ranging between $1$ and $n$.

Let $f=(f_1,\ldots,f_r)$ be an $r$-tuple in $\K[x]^r$.
Consider a free $\K[x,s,\frac{1}{f}]$-module of rank one generated by the formal symbol
$f^s$ and denote it by $M = \K[x,s,\frac{1}{f}]\cdot f^s$.
Here, $s=(s_1,\ldots,s_r)$, $\frac{1}{f} = \frac{1}{f_1\cdots f_r}$
and $f^s = f_1^{s_1}\cdots f_r^{s_r}$. Moreover, we denote by 
$\K\langle S \rangle$ the universal enveloping algebra $U(\mathfrak{gl}_{\,r})$,
generated by the set of variables $S=(s_{ij})$, $i,j=1,\ldots,r$ subject to relations:
$$
  [s_{ij}, s_{kl}] = \delta_{jk} s_{il} - \delta_{il} s_{kj}.
$$
Then, we denote by $D_n\langle S \rangle := D_n \otimes_{\K} \K\langle S \rangle$, which
is a $G$-algebra of Lie type by e.~g. \cite{LS03}. 

The module $M$ has a natural structure of left $D_n \langle S \rangle$-module
when the variables $s_{ij}$ act in the following way $(i\leq j)$:
$$
  s_{ij} \bullet (G(s)\cdot f^s) = s_i \cdot G(s+\epsilon_j-\epsilon_i)\,
  \frac{f_j}{f_i} \cdot f^s\ \in \ M,
$$
where $G(s)$ is an element in $\K[x,s,\frac{1}{f}]$ and $\epsilon_j$ stands
for the $j$-th basis vector.

One can easily observe that the action of $s_{ii}$ on $M$
coincides with the multiplication by $s_i$ from the left.

Following the ideas by Malgrange, one can also consider
$M$ as a $D_n (R) \otimes_{\K} D_r(T)$-module, with $T = \K[t], t=(t_1,\ldots,t_r)$, $\partial t = (\partial {t_1},\ldots,\partial {t_r})$ and the action
\begin{equation}\label{tdtmodule}
\begin{array}{rcl}
  t_i \bullet  (G(s)\cdot f^s) &=& G(s+\epsilon_i)f_j\cdot f^s,\\
  \d {t_i} \bullet  (G(s)\cdot f^s) &=& -s_i G(s-\epsilon_i)\frac{1}{f_i}\cdot f^s.
\end{array}
\end{equation}
Observe that the action of $s_{ij}$ above corresponds to the action of $-\partial {t_i} \cdot t_j$.

\begin{theorem}[\cite*{BMS06}]
For every $r$-tuple $f=(f_1,\ldots,f_r)\in \K[x]^r$ there exists a non-zero
polynomial in one variable $b(s)\in \K[s]$ and $r$ differential operators
$P_1(S),\ldots,P_r(S)\in D_n\langle S \rangle$ such that
\begin{equation}\label{BSvariety}
\sum_{k=1}^r P_k(S) f_k \cdot f^s = b(s_1+\cdots+s_r) \cdot f^s \ \in \ M.
\end{equation}
\end{theorem}

The {\em Bernstein-Sato polynomial} $b_f(s)$ of $f=(f_1,\ldots,f_r)$ is defined to be the monic polynomial of the lowest degree in the variable $s$ satisfying the equation (\ref{BSvariety}). 
It is demonstrated in \cite{BMS06}, that every root of the \BS polynomial
is rational. Let $I$ be the ideal generated by $f_1,\ldots, f_r$ and
$Z$ the (not necessarily reduced) algebraic variety associated with $I$ in $\K^n$.
Then it can be verified
that $b_f(s)$ is independent of the choice of a system of generators of $I$, and moreover
that $b_Z(s) = b_f(s - \codim Z + 1)$ depends only on $Z$. For instance, the
\BS polynomial of $f(x,y)\in \K[x,y]$ and \BS polynomial of the variety defined by the
ideal $\langle f(x,y),z \rangle$ coincide. However, due to the codimension, there is
a shift between $b_f(s)$ and $b_{(f(x,y),z)}(s)$.

Now, let us denote by $\ann_{D_n\langle S\rangle}(f^s)$ the left ideal of all elements
$P(S) \in D_n\langle S \rangle$ such that $P(S)\bullet f^s = 0$. We call this ideal
{\em the annihilator of $f^s$ in $D_n \langle S \rangle$}.
From the definition of the \BS polynomial
it is clear that
$$
  (\ann_{D_n \langle S \rangle} (f^s) + \langle f_1,\ldots,f_r \rangle)\cap
  \K[s_1+\cdots+s_r] = \langle b_f(s_1+\ldots+s_r)\rangle.
$$

Since the final intersection can be computed with the Principal Intersection method \ref{PrincipalIntersect}, 
the above formula provides an algorithm for computing the Bernstein-Sato polynomial
of affine algebraic varieties, once we know a Gr\"obner basis of the annihilator
of $f^s$ in $D_n \langle S \rangle$. The rest of this section is dedicated to
the solving of this problem.

\subsection{The annihilator of $f^s$ in $D_n\langle S \rangle$}

Consider the generalization of Malgrange's ideal $I_f$ associated with $f=(f_1,\ldots,f_r)$,
$$
  I_f = \bigg\langle t_i-f_i\, ,\ \partial_{m} +  \sum_{j=1}^r
  \frac{\partial f_j}{\partial x_m} \partial t_j \left|
  \begin{array}{c} 1\leq i \leq r\\ 1\leq m \leq n \end{array}\right.
  \bigg\rangle \ \subset \ D_n \langle t, \partial t \rangle
$$
Here we give a computer-algebraic proof to
the following Lemma, whose assertion is expected as in \cite{SST00} (for instance).

\begin{lemma} 
\label{MaxAnn}
$I_f$ is a maximal ideal in $D_n(R) \otimes_{\K} D_r (T)$ and
$I_f = \ann_{D_n(R) \otimes_{\K} D_r (T)} f^s$.
\end{lemma}

\begin{proof}
$(t_i - f_i) \bullet f^s = f_i f^s - f_i f^s =0$. For any $m$, observe that
\[
\d_m \bullet (f_1^{s_1} \cdot \ldots \cdot f_{r}^{s_r}) = 
\sum_{j=1}^{r} \d_m \bullet (f_j^{s_j})  (f_1^{s_1} \cdot \ldots \cdot f_r^{s_r}) (f_j^{s_i})^{-1}=
\sum_{j=1}^{r} s_j \frac{\d f_j}{\d x_m} (f_j^{-1})  (f_1^{s_1} \cdot \ldots \cdot f_r^{s_r})
\]
Since $\d t_k$ acts on $f^s$ by the multiplication with $-s_j f_j^{-1}$, the generators
of the second type annihilate $f^s$, so $I_f \subseteq \ann_{D_n(R) \otimes_{\K} D_r (T)} f^s$.

By Lemma \ref{BMgb}, the set
of generators of $I_f$ is the same as the set $S_1$ in the Lemma and hence
there is a monomial ordering, such that $S_1$ is a Gr\"obner basis. In particular,
$I_f$ is a proper ideal. The set of leading monomials of $S_1$ is then $L=\{ t_j, \d_m \}$.
Since any monomial ordering on $\N^{2r+2n}$ can be presented as weighted degree
ordering with the weight vector $w$ with strictly positive entries (see e.~g. \cite{BGV}), 
we see that 
\[
\gkdim (D_n(R) \otimes_{\K} D_r(T))/ I_f = \gkdim \K[\{t_j, \d t_j, x_m, \d_m \}]/
\langle L \rangle = r+n.
\]
Assume the left ideal $I_f$ is not maximal, then there exists $p\not\in I_f$, such
that $I_f \subsetneq I_f + \langle p \rangle\subset D_n(R) \otimes_{\K} D_r (T)$. 
In particular, $\lm(p)$ does not include the elements of $L$ above. If 
$I_f + \langle p \rangle$ is a proper ideal, its set of leading monomials strictly
includes $L$ and has at least one element more. But then the dimension argument
as above shows, that $\gkdim (D_n(R) \otimes_{\K} D_r(T))/ (I_f+\langle p \rangle) 
< r+n$, what contradicts Bernstein's inequality. Hence $I_f$ is maximal and it
is equal to the annihilator.
\end{proof}



\begin{theorem}\label{SannFsVar}
Let $f=(f_1,\ldots,f_r)$ be an $r$-tuple in $\K[x]^r$ and $D_n \langle
\partial t, S \rangle$ the $\K$-algebra generated by $D_n$, $\partial t$
and $S$ with the corresponding non-commutative relations.
Then the following ideal of $D_n \langle S \rangle$ coincides with the
annihilator of $f^s$ in $D_n \langle S\rangle$:
$$
  \bigg[ D_n \langle \partial_t, S \rangle
  \Big( s_{ij} + \d {t_i} f_j \, 
,\ \partial_{m} + \sum_{k=1}^r
  \frac{\partial f_k}{\partial x_m} \partial t_k \left|
  \begin{array}{c} 1\leq i,j\leq r \\ 1\leq m\leq n \end{array}\right.
  \Big)\bigg] \cap D_n \langle S\rangle.
$$
\end{theorem}

\begin{proof}
Let $\phi : D_n \langle S \rangle \hookrightarrow D_n \otimes_{\K} D_r(T)$ 
be the $\K$-algebra homomorphism given by
$\phi(s_{ij}) = -t_j\d {t_i} -\delta_{ij}$ and $\phi (P) = P$ for all $P$ in $D_n$.
In view of Lemma \ref{MaxAnn}, we observe that $\ann_{D_n\langle S \rangle}(f^s) = D_n\langle S\rangle \cap I_f = \phi^{-1}(I_f)$.

The morphism $\phi$ can be written as $\phi = 1_{D_n} \otimes \varphi$, where
$\varphi: \K\langle S\rangle \hookrightarrow D_r(T) = \K\langle t, \d t \rangle$, 
$s_{ij} \mapsto -t_j\d {t_i} -\delta_{ij}$.
Thus we can apply Proposition \ref{preimageCor} to obtain that
$\ann_{D_n\langle S\rangle}(f^s) = (E' I_{\varphi} + E' I_f)\cap D_n\langle S\rangle$,
with $I_{\varphi} = \langle \{ s_{ij}+t_j\d {t_i} +\delta_{ij} \mid 1\leq i,j\leq r \} \rangle$ and
$E' = D_n\langle \{ t_j,\d t_j, s_{ij} \} \rangle$ subject to relations
$$
[s_{ij}, s_{kl}] = \delta_{jk} s_{il} - \delta_{il} s_{kj}, \;
[\d t_k, t_j] = \delta_{jk}, \; [s_{ij}, t_k] = -\delta_{ik} t_j,\; [s_{ij},\d {t_k}]= \delta_{jk} \d {t_i}.
$$
By Theorem \ref{ncPreimage}, $E'$ is a $G$-algebra, if there exists
an elimination ordering for \\$\{\d {t_1},\ldots,\d {t_r}\}$
on $D_n\langle \d t, S \rangle$, obeying the conditions
\[
\lm(\delta_{jk} s_{il} - \delta_{il} s_{kj} ) < s_{ij} s_{kl}, \ \
t_j < s_{ij} t_i,\ \text{ and }\ \d t_i < s_{ij} t_j.
\]
It is clear, that such orderings exist.

Now, we proceed with the elimination of $\{ t_i,\d {t_i} \mid 1\leq i\leq r \}$ 
from $(I_f + I_{\varphi})$ in $E'$. By taking a monomial ordering with the property
$\{t_j \}  \gg  \{ x_i \} $, $\{\d_i, s_{ij} \}  \gg  \{ x_i, \d t_j \}$, we start with eliminating
$\{t_j\}$ first.

By Lemma \ref{BMgb}, the generators $G_1$ of $I_f$ form a Gr\"obner basis. 
The ideal $I_{\varphi}$ in the current situation is generated by $G_2 = \{s_{ij} + t_j \d t_i + \delta_{ij}\}$. In order to prove, that $G_1 \cup G_2$ is a a Gr\"obner basis, 
we apply the generalized Product Criterion (Lemma \ref{prodCrit}). At first we apply reduction process by $G_1$, thus
obtaining $G'_2 = \{s_{ij} + f_j \d t_i\}$.
Then
\[
[s_{ij} + f_j \d t_i, s_{kl} + f_l \d t_k]
= \delta_{jk} (s_{il} + f_l \d t_i) - \delta_{il} (s_{kj} + f_j \d t_k)
\]
which clearly reduces to zero. The next kind of pairs
\[
[s_{ij} + f_j \d t_i, t_k - f_k] = 
- \delta_{ik} t_j +
\delta_{ik} f_j = - \delta_{ik} (t_j -f_j)
\]
again reduces to zero. It remains to consider
\[
[s_{ij} + f_j \d t_i, \d_m + \sum_{k=1}^r \frac{\d f_k}{\d x_m} \d {t_k}] =
%
\sum_{k=1}^r \frac{\d f_k}{\d x_m} \delta_{jk} \d {t_i} -
\d t_i [\d_m, f_j] = 0,
\]
since $[\d_m, f_j] = \frac{\d f_j}{\d x_m}$.

Hence, $G_1 \cup G'_2$ is a a Gr\"obner basis and hence, by the
Elimination Lemma, $G_3 = (G_1 \cup G'_2) \setminus\{t_j-f_j\}$
is a Gr\"obner basis of $(I_f + I_{\varphi}) \cap D_n \langle \d t_k, s_{ij} \rangle$.
Thus, it follows, that 
\[
\ann_{D\langle S \rangle} (f^s) =
\Big\langle s_{ij} + \d {t_i} f_j \, , \, \d_m + \sum_{k=1}^r \frac{\d f_k}{\d x_m} \d {t_k} \Big\rangle \cap D_n \langle S\rangle.\qedhere
\]
\end{proof}

Indeed, the result we have proved is a natural generalization
of the algorithm for computing the annihilator of $f^s$ in $D_n[s]$ (cf. \ref{BM}) given by
Brian\c con-Maisonobe in \cite{BM02}. Finally, the algorithm for the computation
of $\ann_{D_n\langle S \rangle} f^s$ looks as follows:


\begin{algorithm}[\texttt{SannfsVar}] ~
\begin{algorithmic}
\REQUIRE $f = (f_1,\ldots,f_r)$, an $r$-tuple in $\K[x]^r$
\ENSURE $\{G_1(S),\ldots, G_e(S)\}$, a Gr\"obner basis of
  $\ann_{D_n \langle S\rangle}(f^s)$
\STATE Let $D_n\langle \partial t, S\rangle$ be the algebra in Corollary \ref{SannFsVar},
with non-commutative relations
$$
[\partial_i,x_i] = 1,\quad [s_{ij},\d {t_k}]= \delta_{jk} \d {t_i},\quad
[s_{ij}, s_{kl}] = \delta_{jk} s_{il} - \delta_{il} s_{kj}.
$$
\STATE $J_1 := \big\langle \{ s_{ij}+\partial {t_i}f_j \mid 1\leq i,j \leq r \} \big\rangle$
\STATE $J_2 := \big\langle \{ \partial_{m} + \sum_{k=1}^r \frac{\partial f_k}{\partial x_m}
    \partial_{t_k} \mid 1\leq m\leq n \} \big\rangle$
\STATE $J:= J_1 + J_2$
  \hfill $\triangleright$ $J\subseteq D_n\langle \partial t, S \rangle$
\STATE $H:=$ G.B. of $J$ w.r.t. a compatible elim. ordering for $\partial {t_1},
     \ldots,\partial {t_r}$
\STATE $H\cap D_n\langle S \rangle =: \{G_1(S),\ldots, G_e(S)\}$
\RETURN $\{G_1(S),\ldots, G_e(S)\}$
\end{algorithmic}
\end{algorithm}

\subsection{Elimination orderings in $D_n\langle \partial t, S\rangle$}\label{elimORD}

One of the bottlenecks of the presented algorithm for computing the Bernstein-Sato
polynomial for varieties is to calculate the corresponding annihilator.
An elimination term ordering for $\{\partial t_1,\ldots,\partial t_r\}$ in
$D_n\langle \partial t,S\rangle$, which has
$2n+r+r^2$ variables, has to be considered. In addition, due to the structure of the $G$-algebra,
this ordering $<$ has to be chosen with the following extra restrictions
$$
  \partial t_i < s_{ij} \partial t_j,\qquad
  \lm(\delta_{jk} s_{il} - \delta_{il} s_{kj}) < s_{ij} s_{kl}
$$
for all indices $i,j,k,l$ where the expression makes sense. The efficiency of
the method strongly depends on the selected ordering. Therefore it is worth
analysing it in detail.

Assume that $<$ is such an ordering and let us consider the first two rows of
the matrix representing the ordering in this way.
$$
  \begin{array}{cccccc}
  \partial {t_1} & \cdots & \partial {t_r} & S & x & \partial_x\\
  \hline
  p_1 & \cdots & p_r & {\bf a} & {\bf b} & {\bf c}\\
  q_1 & \cdots & q_r & \alpha & \beta & \gamma\\
  \hline
  &&& <'\\
  \hline
  \end{array}
$$

The vectors ${\bf a}$, ${\bf b}$ and ${\bf c}$ must be zero, since $<$ is an
elimination ordering for $\{\partial t_i\}$. The conditions $\partial t_i < s_{ij}
\partial t_j$, imply $p_i\leq p_j$ for all $i,j$. Thus all $p_1,\ldots,p_r$ are
equal and can be taken as 1.

From computational point of view, since the variables $\{s_{ij}\}$ do not commute
with $\{\partial t_i\}$, these two blocks must be together in the elimination
ordering, namely $\beta = \gamma = 0$, otherwise Gr\"obner bases computation
may be slow.

In the implementation we have taken $\alpha_{ii}=2$ and $\alpha_{ij}=1$ for $i\neq j$,
and ${\bf q}=0$.
However, in some examples we have observed that lexicographical orderings are also
useful, see Example \ref{exLexic} below. 


In this section we have described an algorithm for computing the Bernstein-Sato polynomial
of affine algebraic varieties without any homogenization but passing through the
computation of the annihilator of $f^s$ in $D_n\langle S\rangle$. Now, other methods
are ilustrated.

\subsection{Another approach}

As Budur et.~al. point out in \cite[p.\ 794]{BMS06}, the Bernstein-Sato
polynomial for varieties coincides, up to shift of variables, with the $b$-function in \cite[p.\ 194]{SST00}, if the weight vector is chosen appropriately. Let us describe this algorithm more carefully.

Let $I_f = \ann_{D_n\langle t, \partial t \rangle} (f^s)$ be the Malgrange ideal associated with $f=(f_1,\ldots,f_r)$ and consider the weight vector $w = ((0,\ldots,0),(1,\ldots,1) \in \mathbb{Z}^n \times \mathbb{Z}^r$ which gives weight $0$ to $\partial_m$ and weight $1$ to $\partial t_i$. Consider the $V$-filtration $V = \{ V_k \mid k\in \Z \}$ on $D_n\langle t,\partial t \rangle$ with respect to $w$, where $V_k$ is spanned by $\{ t^{\alpha}\cdot \partial t^{\beta} \mid -|\alpha|+|\beta| \leq k \}$ over $\K$. Note that the associated graded ring $\oplus_{k\in\mathbb{Z}} V_k/V_{k-1}$ is isomorphic again to the $(n+r)$-Weyl algebra $D_n \langle t,\partial t \rangle$ and the homogeneous parts are the following.
$$
V_k / V_{k-1} =
\begin{cases}
D_n \langle t_i\cdot \partial t_j \rangle \partial^{\beta}, & |\beta| = k > 0;\\
D_n \langle t_i\cdot \partial t_j \rangle, & k=0;\\
D_n \langle t_i\cdot \partial t_j \rangle t^{\alpha}, & -|\alpha| = k < 0.
\end{cases}
$$

Denote by $B(s)$ the $b$-function of the holonomic ideal $I_f$ with respect to~$w$. Recall that $B(s)$ is the monic generator of the ideal $\ini_{(-w,w)}(I_f) \cap \K[t_1 \partial t_1 + \cdots + t_r \partial t_r ]$.

As in the classical case, i.e. $r=1$, the following result holds.

\begin{lemma}
$b_f(s) = (-1)^{\deg B(s)} B(-s-r)$.
\end{lemma}

\begin{proof}
Consider $P_1(S), \ldots, P_k(S)\in D_n \langle S \rangle$ differential operators satisfying the functional equation $\sum_{k=1}^r P_k(S) f_k \bullet f^s = b_f(s_1 + \cdots + s_r)\bullet f^s$. Then $b_f(s_1+\cdots+s_r) - \sum_{k=1}^r P_k(S) f_k$ is an element in $\ann_{D_n\langle S \rangle} (f^s)$ and hence applying the Mellin transform, or equivalently making the substitution $s_{ij}\mapsto -t_j \partial t_i - \delta_{ij}$, one obtains the following element in $I_f$.
$$
  b_f(-t_1 \partial t_1 - \cdots - t_r \partial t_r - r) - \sum_{k=1}^r P_k (- t_j \partial t_i - \delta_{ij}) f_k \ \in \ I_f
$$
Modulo $I_f$ the polynomials $f_k$ in the above expresion can be replaced by $t_k$, since $t_k-f_k\in I_f$. Finally, taking initial parts one concludes that $b_f(- t_1 \partial t_1 -  \cdots - t_r \partial t_r - r) \in \ini_{(-w,w)}(I_f)$, which means that $B(s)$ divides $b_f(-s-r)$.

Conversely, by definition there exists a differential operator $P(t,\partial t) \in I_f \subset D_n \langle t, \partial_t \rangle$ such that $B(t_1 \partial t_1 + \cdots + t_r \partial t_r) = \ini_{(-w,w)}(P(t,\partial t))$. In particular $P(t,\partial t)$ has $V$-degree zero and hence it can be decomposed into $V$-homogeneous parts as follows
$$
  P(t,\partial t) = B(t_1 \partial t_1 + \cdots + t_r \partial t_r) + \sum_{|\alpha| \geq 1} Q_{\alpha} (t_i \partial t_j) t^{\alpha}.
$$
Since $P(t,\partial t ) \in I_f$, making left reduction of $P(t,\partial t)$ with respect to $\{t_i - f_i\}$ we arrive at $B(t_1 \partial t_1 + \cdots + t_r \partial t_r) + \sum_{|\alpha| \geq 1} Q_{\alpha} (t_i \partial t_j) f^{\alpha} \in I_f \cap D_n \langle t_i \partial t_j \rangle$. After applying the substitution $t_i \partial t_j\mapsto -s_{ji} - \delta_{ij}$, we conclude that $B(-s_1-\ldots-s_r-r)$ belongs to the ideal $\ann_{D_n \langle S \rangle}(f^s) + \langle f_1,\ldots,f_r \rangle$ and the proof is complete.
\end{proof}


Algorithms for computing this $b$-function, which use the homogenization technique in the Weyl algebra,
are given in Section \ref{bfctIdeal}, see also \cite{SST00}.
We describe the complete algorithm for computing Bernstein-Sato polynomials using
initial parts.


\begin{algorithm}[\texttt{bfctVar}] ~
\begin{algorithmic}
\REQUIRE $f = (f_1,\ldots,f_r)$, an $r$-tuple in $\K[x]^r$;
Z, variety associated with $f$
\ENSURE $b_Z(s) = b_f(s-\codim Z+1)$, Bernstein-Sato polynomial of $Z$
\STATE Let $D_n\langle t, \partial t \rangle = D_n \otimes D_r$ be the $(n+r)$-Weyl algebra.
\STATE $I := \big\langle \big\{t_i-f_i\big\}_{i=1}^r\, ,\
  \big\{\partial_{m} +  \sum_{j=1}^r
  \frac{\partial f_j}{\partial x_m} \partial{t_j}\big\}_{m=1}^n
  \big\rangle$
  \hfill $\triangleright$ $I\subseteq D_n\langle t, \partial t\rangle$
\STATE $w:=((0,\ldots,0),(1,\ldots,1))\in\Z^n\times \Z^r$
\STATE $J:=\texttt{InitialIdeal}(I,w)$
  \hfill $\triangleright$ Algorithm \ref{InitialIdeal}
\STATE $s:=- (\partial t_1\cdot t_1 +\ldots+ \partial t_r\cdot t_r)$
\STATE $b(s):=\texttt{pIntersect}(s,J)$
  \hfill $\triangleright$ Algorithm \ref{PrincipalIntersect}\\
\RETURN $b(s-\codim Z + 1)$
\end{algorithmic}
\end{algorithm}

As for the \BS polynomial of a polynomial (indicated in Remark \ref{2methods}),
we have two different methods for the computation of \BS polynomial of an affine algebraic variety, namely
\begin{itemize}
\item minimal polynomial for $s_1+\ldots+s_r$ in $D/ \ini_{(-w,w)} (I_f)$,
\item $(\Ann f^s + \langle f_1,\ldots, f_r \rangle) \cap \K[s_1+\ldots+s_r]$, where the intersection can be done rather with the \ref{PrincipalIntersect} method, than by using Gr\"obner basis elimination.
\end{itemize}

It is important to investigate the connection of these methods and especially their applicability
to different classes of varieties. Our experience shows, that no method is clearly superior
to the other one in general. Thus it is desired to have both of them in any package
for $D$-modules.

\begin{remark}
Very recently, the authors have realized another approach for computing Bernstein-Sato
polynomials for varieties. In \cite{Shibuta08}, Shibuta modifies the definition of
Budur-Musta{\c{t}}{\v{a}}-Saito's Bernstein-Sato polynomials to determine a system of
generators
of the multiplier ideals of a given ideal. Then he obtains an algorithm for computing
Bernstein-Sato polynomials, which gives an algorithm for computing multiplier ideals and
jumping coefficients.
His methods are based on the theory of Gr\"obner bases in Weyl algebras and corresponds
to the natural generalization given by Oaku and Takayama, hence they need
homogenization techniques.
\end{remark}

We conclude this section showing several examples calculated with our experimental implementation.

\newcommand{\B}{\Big}
\begin{example}\label{exLexic}
Let $TX = V(x_0^2+y_0^3, 2x_0 x_1 + 3 y_0^2 y_1) = V(f_1,f_2) \subset \C^4$ the tangent bundle
of $X = V(x^2+y^3)\subset \C^2$. The annihilator of $f^s$ in
$D\langle S\rangle$  and the Bernstein-Sato polynomial of $TX$ using the previous approach
can be computed with the {\sc Singular} commands {\tt SannfsVar} and {\tt bfctVarAnn}.
\begin{verbatim}
LIB "bfunVar.lib"; 
ring R = 0,(x0,x1,y0,y1),Dp;
ideal F = x0^2+y0^3, 2*x0*x1+3*y0^2*y1;
bfctVarAnn(F);
\end{verbatim}
The output is lengthy, hence we supress it. We obtain an ideal called {\tt LD} with 15 generators and the \BS polynomial for $TX$, which looks as follows
$$
	b_{TX}(s) = (s+1)^2 (s+\frac{1}{3})^2 (s+\frac{2}{3})^2 (s+\frac{1}{2}) (s+\frac{5}{6}) (s+\frac{7}{6}).
$$
Analogously, one can consider the tangent bundle of $V(x^4+y^5)$. In this case the Bernstein
polynomial has degree 42 and it equals
$$
\begin{array}{l}\displaystyle
\B(s+\frac{1}{5}\B)^2 \B(s+\frac{3}{5}\B)^2 \B(s-\frac{1}{5}\B)^2
\B(s+\frac{2}{5}\B)^2 \B(s+\frac{1}{15}\B) \B(s+\frac{1}{4}\B)
\B(s-\frac{4}{15}\B) \B(s+\frac{1}{6}\B)\\[0.25cm] \displaystyle
\B(s+\frac{13}{20}\B) \B(s-\frac{1}{12}\B) \B(s+\frac{8}{15}\B)
\B(s+\frac{1}{20}\B) \B(s-\frac{1}{20}\B) \B(s+\frac{1}{2}\B)
\B(s+\frac{7}{20}\B) \B(s+\frac{4}{15}\B)\\[0.25cm] \displaystyle
\B(s+\frac{2}{3}\B) \B(s+\frac{2}{15}\B) \B(s+\frac{5}{12}\B)
\B(s+0\B) \B(s-\frac{1}{6}\B) \B(s-\frac{1}{15}\B) \B(s-\frac{1}{10}\B)
\B(s-\frac{5}{12}\B)\\[0.25cm] \displaystyle
\B(s-\frac{2}{15}\B) \B(s-\frac{3}{20}\B) \B(s-\frac{3}{10}\B) \B(s+\frac{7}{15}\B)
\B(s+\frac{3}{20}\B) \B(s+\frac{1}{3}\B) \B(s+\frac{3}{10}\B) \B(s-\frac{1}{3}\B)\\[0.25cm] \displaystyle
\B(s+\frac{1}{10}\B) \B(s+\frac{11}{20}\B) \B(s+\frac{9}{20}\B) \B(s+\frac{1}{12}\B)
\B(s+1\B) \B(s+\frac{3}{4}\B).
\end{array}
$$
The result was not able to be obtained using the elimination ordering given
in section \ref{elimORD}. Instead the following monomial ordering (in
{\sc Singular} format) has been taken: 
\begin{verbatim}
     "(a(1,1),a(0,0,2,1,1,2),(dp(6),rp)".
\end{verbatim}

Note that if $TX$ is the variety defined by $f=(f_1,f_2)$, then $b_f(s)$ has always negative roots.
As this examples shows the same is not true for $b_Z(s)$, due to the codimension.
\end{example}

\begin{example}
Let $Z$ be the algebraic variety defined by
$f= (x_1^3-x_2 x_3, x_2^2-x_1 x_3, x_3^2-x_1^2 x_2)$. Then
$$
  b_Z(s) = \B(s+1\B)^2 \B(s+\frac{7}{9}\B) \B(s+\frac{5}{9}\B) \B(s+\frac{1}{2}\B)
 \B(s+\frac{8}{9}\B) \B(s+\frac{11}{9}\B) \B(s+\frac{10}{9}\B) \B(s+\frac{4}{9}\B).
$$
This is actually Example 5.10 in \cite{Shibuta08} and it corresponds to the
space of monomial curve 
$\mathrm{Spec}\,\C[T^3,T^4,T^5] \subset \C^3$.
Note that the $b$-function coincides with the one that appears in \cite{Shibuta08},
since we are computing $b_Z(s)$ instead of $b_f(s)$.
\end{example}


\begin{example}
Let $Z$ be the so-called Hirzebruch-Jung singularity of type $(5,2)$. It is a cyclic
quotient singularity and can be seen as the algebraic variety associated with
the ideal $\langle z_3^2-z_2 z_4, z_2^2 z_3-z_1 z_4, z_2^3-z_1 z_3 \rangle
\subset \C [z_1, z_2, z_3, z_4]$. Then
$$
  b_Z(s) = (s+1)^3 (s+\frac{4}{3}) (s+\frac{5}{3}) (s+\frac{3}{2}).
$$
\end{example}

\begin{example}
As an intractible example we would like to mention the following one. Let $Z$ be
the cyclic quotient singularity of type $(6;1,2,3)$. It can also be seen as the
toric variety associated with the matrix
$$
  A := \left(\begin{array}{ccccccc}
  6 & 4 & 2 & 0 & 3 & 1 & 0\\
  0 & 1 & 2 & 3 & 0 & 1 & 0\\
  0 & 0 & 0 & 0 & 1 & 1 & 2
  \end{array}\right).
$$


The corresponding ideal can be taken to have 9 generators in $\C[z_1,\ldots,z_7]$. 
We are not yet able to compute the Bernstein-Sato polynomial even trying several elimination orderings.
\end{example}

\section{Conclusion and Future Work}

From our recent articles there follow some important conclusions.
\begin{enumerate}
\item In \cite{LM08} we proved, that for the computation of
Bernstein-Sato polynomial of a hypersurface, the homogenization, used in 
the method of Oaku and Takayama is superfluous. This, however, does not apply to
the situation of computing a $b$-function with respect to weights. 
By several steps in \cite{LM08, ALM09} and in this article we have shown, that the method
by Brian\c{c}on and Maisonobe can be seen as natural refinement of the  
method of Oaku and Takayama. 
\item From our investigations it follows, that we cannot in general
prove, that either initial-based method or annihilator-based one
for the computation of Bernstein-Sato polynomial is definitely more efficient than
the other one. Instead, on numerous examples we see that roughly the domain of
better performance of initial-based method includes hyperplane arrangements, while
for other singularities annihilator-based method scores distinctly better. 
It is important to continue these investigations and derive more classes of
polynomials, when possible, in order to use this information in attempts
to estimate at least the practical complexity of $D$-module computations.
\item Also for the syzygy-driven method to compute $\Ann_{D[s]} (f^s)$ 
we do not have yet a proof of its superiority over the method
by Brian\c{c}on and Maisonobe. Due to the reasons we explain in this
paper one could achieve such superiority. But on the other hand, 
there are examples, which show the contrary. Even if there are only
a few examples of such kind, it is interesting to investigate this
phenomenon deeper.
\item We pay so much attention to the algorithms for the case
of a hypersurface due to many reasons. According to our proofs
for the case of a variety, we use indeed the same technology. We
expect to generalize all of enhancements we have described to the
case of an affine variety in a similar way we presented the generalization
of algorithms for annihilator and Bernstein-Sato polynomial. It is very
important to stress, that working in the case of a variety is a priori 
much more involved computationally, hence more attention on very
effective algorithms need to be paid.  
\end{enumerate}

\begin{ack}
We would like to thank Francisco Castro-Jim\'{e}nez, Jos\'{e}-Mar\'{i}a Ucha, Gert-Martin Greuel,
Enrique Artal and Jos\'{e}-Ignacio Cogolludo for their constant support and motivation in our
work over years.

Special thanks go to Hans Sch\"onemann and Michael Brickenstein, whose continuous support 
in numerous aspects of \textsc{Singular} and discussions on computer algebra we deeply appreciate.

We are grateful to Masayuki Noro
and Anton Leykin for their help and careful explanations, concerning the computer algebra systems \textsc{Asir} and \textsc{Macaulay2}. 

We also thank Uli Walther for fruitful discussions.
\end{ack}


\end{document}